\documentclass[a4paper,two-side,10pt]{article}
\usepackage[english]{babel}
\usepackage[latin1]{inputenc}
\usepackage[T1]{fontenc}
\usepackage{amsmath}
\usepackage{amssymb}
\usepackage{aeguill}
\usepackage{graphicx}
\usepackage[all]{xy}
\usepackage{eucal}
\usepackage{multicol}
\newcommand{\E}{\mathbb{E}}
\newcommand{\R}{\mathbb{R}}
\newcommand{\C}{\mathbb{C}}

\newcommand{\ds}{\mathbb{S}}

\newcommand{\eps}{\varepsilon}
\newcommand{\ud}{\text{d}}
\newcommand{\D}{{\rm{D}\hskip-0.7em / \,}}			
\newcommand{\gu}{\hspace{-0.3cm}\phantom{A}^1\tilde{\text{G}}}
\newcommand{\gd}{\hspace{-0.3cm}\phantom{A}^2\tilde{\text{G}}}
\newcommand{\Uu}{\hspace{-0.3cm}\phantom{A}^1\tilde{\text{U}}}
\newcommand{\Ud}{\hspace{-0.3cm}\phantom{A}^2\tilde{\text{U}}}
\newcommand{\fin}{$\bigcirc \hspace{-0.31cm}\divideontimes$}

\newcommand{\Dd}{\mathcal{D}(\ds)}

\newcommand{\proof}{\bfseries \underline{Proof} \mdseries}

\numberwithin{equation}{section}
\newtheorem{theorem}{Theorem}[section]  
\newtheorem{remark}[theorem]{Remark}  
\newtheorem{lemma}[theorem]{Lemma}

\newtheorem{definition}[theorem]{Definition}  
\newtheorem{proposition}[theorem]{Proposition}

\reversemarginpar
\date{}
\begin{document}
\title{Integral Formula for the Characteristic Cauchy Problem on a curved Background}
\author{Jérémie Joudioux\footnote{tel: + 33 98 01 72 85 -- mail: jeremie.joudioux@univ-brest.fr }\\
Laboratoire de Mathématiques de Brest, U.M.R. CNRS 6205\\ 6, avenue Victor Le Gorgeu, CS 93837, F-29238 BREST Cedex 3}
\maketitle

\begin{quote}
Abstract: We give a local integral formula, valid on general curved space-times, for the characteristic Cauchy problem for the Dirac equation with arbitrary spin using the method developed by Friedlander in \cite{Friedlander:1975vn}. The results obtained by Penrose in the flat case in \cite{p80} are recovered directly. It is expected that this method can be used to obtain sharp estimates for the characteristic Cauchy problem for the Dirac equation.\\
Résumé: Nous donnons une formule intégrale pour le problème de Cauchy caractéristique local pour l'équation de Dirac pour le spin arbitraire en utilisant la méthode développée par Friedlander dans \cite{Friedlander:1975vn}. Nous retrouvons alors directement le résultat de Penrose dans le cas plat (\cite{p80}).
\end{quote}

-------------------------------

Penrose obtained in 1963 (\cite{p80}) an integral formula for the characteristic Cauchy problem for the Dirac equations for arbitrary spin in the flat case. His derivation of the integral formula is based on the construction of a Newman-Penrose tetrad (null tetrad) adapted to the null structure of the null initial data hypersurface, and especially to the description of the behavior of the null generators (bicharacteristics) of the cone. The use of the 2-spinor formalism allows him to write the solution of the problem in function of a "null datum", contraction of the data on the cone with its spinor generators. The formula is verified a posteriori through a splitting of the Dirac operator over the spin basis in the compacted spin coefficient formalism. Penrose expected that this formula could be extended to the analytic case. As far as the author knows, the general case remains open.

Friedlander gave in the mid 70's (\cite{Friedlander:1975vn}) a method to obtain a parametrix for the wave equation derived from the Leray constructions (see for instance \cite{gkl64}) and wrote an integral representation of the solution of the characteristic Cauchy problem. His construction is based on a natural decomposition of the fundamental solution on the cone. Another approach exists to the characteristic Cauchy problem based on Fourier Integral operators. It must furthermore be noticed that there is no general result about the characteristic Cauchy problem for hyperbolic operators. Hörmander gave in \cite{MR1073287} a general result of existence and uniqueness, together with energy estimates, for the wave equation on a spatially compact Lorentzian manifold.

The purpose of this paper is to combine the method developed by Friedlander with the description of the null cone by Penrose to obtain an integral formula for the characteristic Cauchy problem with initial data on the cone for arbitrary spin in general curved spacetimes. The choice of this method implies that we face the same restrictions as in the book by Friedlander. There exists an essential obstacle to the extension of the domain of validity of the representation formula: the existence of caustics which limit the domain where the formula can be written. We then have to restrict ourselves to a geodesically convex domain $\Omega$ of a smooth Lorentzian manifold $(M, g)$, that is to say a domain where there exists a unique geodesic between any pair of distinct points. This restriction is inherent to the method and the fact we work with arbitrary curved geometry. The advantage is however that we obtain an explicit integral formula without resorting to any microlocalization. This in principle should allow an extension to metrics of low regularity in the spirit of \cite{MR2339803}.

 More explicitly, let us consider $(M,g)$ a smooth Lorentzian manifold and $p_0$ a point in $\Omega$; the problem:
\begin{equation*}
\D u=0  
\end{equation*}
where $u$ is a section of a given fiber bundle on $\Omega$ and $\D$ is the Dirac operator on this bundle, with the initial conditions on the future null cone $\mathcal{C}^+(p_{0})$:
\begin{equation*}
u=\theta \textrm{ on } \mathcal{C}^+(p_0)\cap \Omega
\end{equation*} 
is known as a first order Goursat problem with initial data on the characteristic hypersurface $\mathcal{C}^+(p_0)\cap\Omega$.

It is known that several conditions must be satisfied to ensure that this problem admits a solution. The first one comes from a geometric obstruction to the existence of a solution when symmetry conditions on the field $u$ are imposed; this implies that the manifold $M$ must satisfy some geometric assumptions, known as the consistency conditions, depending on the spin we are working with. The second one comes from the fact that the initial data are given on a characteristic hypersurface: $\theta$ must then satisfy the restriction of the Dirac equation to the cone from $p_0$:
\begin{equation*}
\D|_{\mathcal{C}^{+}(p_{0})}\theta=0.
\end{equation*}
These equations are called the compatibility equations for the initial data.

As already mentioned, there exists, as far as the author knows, no general result about the characteristic Cauchy problem. Nonetheless, it is worth mentioning some results of existence and uniqueness with some generality. In the analytical case, this problem is similar to the Cauchy-Kowaleski problem (see for instance \cite{gkl64}). The problem is well posed in that case. This can be extended, with energy estimates, to minimal regularity (\cite{arXiv:0903.0515v1}). The well-posedness of the characteristic Cauchy problem is nonetheless not the point of this paper: assuming existence and uniqueness of the solution in the neighborhood of the point $p_0$, the goal consists in deriving a representation formula for this solution.

The paper is organised as follows. The first part presents an adaptation of the Friedlander method to the bundle of Dirac spinors. After a geometric and intrinsic presentation of the theory of spinors, the analytic tools to write a fundamental solution of the Dirac equation are developed. The second part is devoted to the derivation of the formula for Dirac spinors. Following Penrose's construction, a null tetrad adapted to the structure of the null cone is constructed and used to describe the geometric tools. The integral formula can then be derived from the parametrix and the result obtained by Penrose is recovered for Weyl (or two-) spinors. Finally, the third part deals with the arbitrary spin $\frac{n}{2}$. The presentation made in the first part is adapted to the bundle of spinors with spin $\frac{n}{2}$ so that the construction can be applied directly. A representation formula is then given for arbitrary spin and simplified in the case of the Maxwell equations. Penrose's formula for the characteristic Cauchy problem for arbitrary spin in the flat case is recovered in a flat spacetime.

The author would like to thank his supervisor, Jean-Philippe Nicolas, for his kind and patient support, his advices and his culture.

\tableofcontents

\bfseries Notations and conventions. \mdseries

We describe here for future reference the notations and conventions which will be used all along the paper. Note that smooth means $C^\infty$ in this paper.
\begin{enumerate}
\item Geometric notations:
\begin{enumerate}
\item General framework:
\begin{itemize}
\item $(M,g)$: smooth Lorentzian oriented and time oriented manifold with a metric $g$ having signature signature $(+,-,-,-)$;
\item $\Omega$: geodesically convex domain of $M$;
\item $\mu$: volume form associated with the metric $g$ on $M$;
\item $p_0$: a given point in $\Omega$;
\item $\nabla$: Levi-civita connection for $g$ on the tangent bundle of $M$, $TM$.
\end{itemize}
\item Null structure on $\Omega$: let $p$ be a given point in $\Omega$:
\begin{itemize}
\item $\mathcal{C}(p)$: null cone from $p$, that is to say the set of points of $\Omega$ which lie on a null geodesic passing through $p$; 
\item $\mathcal{C}^+(p)$ (resp. $\mathcal{C}^-(p)$): future (resp. past) null cone from $p$, that is to say the set of points of $\Omega$ that lie on a future (resp. past) oriented null geodesic from $p$; 
\item $\mathcal{I}(p)$: chronological set from $p$, that is to say the points of $\Omega$ which lie on a timelike or null geodesic passing through $p$; 
\item $\mathcal{I}^+(p)$ (resp. $\mathcal{I}^-(p)$): future (resp. past) chronological set from $p$, that is to say the set of points of $\Omega$ that lie on a future (resp. past) oriented timelike or null geodesic from $p$;
\item $\mathcal{J}(p)=\mathcal{I}(p)\backslash \mathcal{C}(p)$: causal set from $p$ and $\mathcal{J}^\pm(p)=\mathcal{I}^\pm(p)\backslash \mathcal{C}^\pm(p)$ are the future and past causal sets from $p$.
\end{itemize}
\item Spin structure: $\Omega$ is endowed with a spin structure; the spinors will be denoted using the Penrose conventions as well as the usual algebraic notations according to convenience:
\begin{itemize}
\item $\ds_{Dirac}$: fibre bundle of Dirac (or 4-) spinors; 
\item  $\ds_A$ and $\ds^{A'}$: bundles of Weyl (or 2-) spinors (resp. dual and anti-spinors); 
\item "$\cdot$": Clifford multiplication;
\item $(\cdot,\cdot)$: symplectic product on $\ds_{Dirac}$ obtained by lifting the metric $g$; 
\item $\epsilon^{AB}$ and $\epsilon_{A'B'}$: restrictions of $(\cdot,\cdot)$ to $\ds_A$ and $\ds^{A'}$;
\item $\C^\infty_0(\ds_{Dirac})=\mathcal{D}(\ds_{Dirac})$: smooth sections with compact support in $\Omega$ endowed with the usual Fréchet topology;
\item $\mathcal{D}'(\ds_{Dirac})$: its topological dual; 
\item $\C^\infty(\ds_{Dirac})=\mathcal{E}(\ds_{Dirac})$: smooth sections of $\ds_{Dirac}$ on $\Omega$,;
\item $\mathcal{E}'(\ds_{Dirac})$: its topological dual; 
\item the connection $\nabla$ on $T\Omega$ is lifted on $\ds_{Dirac}$ and is still denoted $\nabla$;
\item the Dirac operator is defined, for a given section $(e_i)_{i\in \{0, \dots, 3\}}$ of the fibre bundle of orthonormal frames, on $C^\infty(\ds_{Dirac})$ by:
$$
\forall \Phi \in C^\infty(\ds_{Dirac}), \D\Phi=\sum_{i\in \{0, \dots, 3\}}e_i\cdot \nabla_{e_i}\Phi 
$$
\end{itemize}
\end{enumerate}
\end{enumerate}

\section{Geometric and analytic preliminaries}

The geometric and analytic tools are presented in this section. As already mentioned in the introduction, due to geometric obstructions such as conjugates points or convergence of geodesics, the whole paper restricts itself to a geodesically convex domain $\Omega$:
\begin{definition}
A domain $\Omega$ is said to be geodesically convex if and only if it is an open set where, for every pair of points $(p,q)$ in $\Omega$, there exists a unique geodesic between $p$ and $q$.
\end{definition}

\subsection{Dirac spinors and Dirac equation}\label{sectionspinor}
\noindent This section presents a construction of the spinor bundle so that it will be possible to apply the method of Friedlander in the most direct way. This presentation also intends to be a small dictionary between an abstract presentation of the theory of spinors and the Penrose conventions to represent spinors in terms of indices. Finally it must be noticed that, though the presentation is made on $\Omega$, it can be generalized to a globally hyperbolic manifold (see remark \ref{geroch} below).

\subsubsection{Abstract construction}
We begin by defining a spin bundle:
\begin{definition} 
A manifold $M$ is said to be spin if its tangent bundle admits a spin structure, that is to say there exists a $\text{Spin(1,3)}$ principal bundle $P_{\ds}$, together with a twofold covering  $\xi: P_\ds\rightarrow P_{SO}M$, where $P_{SO}M$ is the $SO(1,3)$-principle bundle of orthonormal frames on $M$, such that
$$
\forall p\in P_{\ds}, \forall g\in \text{Spin}(1,3), \xi(pg)=\xi(p)\xi_0(g).
$$
where $\xi_0$ is the universal covering from $\text{Spin}(1,3)\approx SL_2(\C)$ on $SO(1,3)$.
\end{definition}

\begin{remark}\label{geroch}
\begin{enumerate}
\item The existence of a spin structure on a manifold is usually ensured by the assumption that its second Stiefel Whitney class vanishes. 
\item In the case of a four dimensional Lorentzian manifold $(M,g)$, Geroch showed in \cite{g68} that a necessary and sufficient condition for $M$ to carry a spin structure is that its bundle of orthonormal frames admits a global section (this is referred to as parallelizability).
\item A common assumption in general relativity which ensures that a 4-dimensional Lorentzian manifold is spin is the global hyperbolicity assumption: there exists in $M$ a global Cauchy hypersurface, i.e. a spacelike  hypersurface such that any inextendible timelike geodesic intersects this surfaces exactly once(\cite{g68,g70}).  
\end{enumerate}
\end{remark}

The spinor bundle on $\Omega$ is defined through the action of an algebra over a vector space. This construction requires the following tool, which consists in group action over a fibre bundle, replacing its previous fibre by a given vector space: 
\begin{definition}
Let $(E,\Omega,\pi)$ be a $G$-principal bundle. Let $F$ be a vector space and\\ $\rho: G\rightarrow Homeo(F)$ a continuous map.\\
Consider the action 
\begin{equation*}
\begin{array}{cccc}
\phi:& G&\rightarrow &Aut(E)\times Homeo(F)\\
&g&\mapsto&\big((x,y)\in E\times F\mapsto (xg^{-1}, \rho(g)y)\big)
\end{array}.
\end{equation*}
The quotient space
$$
E\times F\big/ \rho \text{ or }E\times_{\rho} F
$$
with projection $\tilde\pi$ obtained by factorization of the diagram
\begin{equation*}
\xymatrix{E \times F \ar[rr]^{\pi\circ p} \ar[dr]_{\phi}& & M\\
        &E\times_{\rho} F\ar[ur]_{\tilde\pi}}
\end{equation*}
where $p: E\times F\rightarrow F$ is the projection on the first variable, is a $G$-principal bundle with fibre $F$.
\end{definition}

The algebra, known as the Clifford algebra associated with a given quadratic form, which is used to construct the spinor bundle is then defined:
\begin{definition}
Let $E$ be a vector space (real or complex) with a quadratic form $q$. The Clifford algebra $(Cl(E),+, \cdot)$ is the quotient space: 
\begin{equation*}
Cl(E,q)=\left(\bigoplus_{n=0}^{+\infty}\bigotimes^nE\right)\big/ I(E)
\end{equation*}
where $I(E)$ is the ideal generated by the set $\{v\otimes v-q(v)| v\in E\}$.
\end{definition}
This algebra is known to have the following structure (\cite{Lawson:1989lr}):
\begin{proposition}
There exist two sub-algebrae denoted $Cl^0(E,q)$ and $Cl^1(E,q)$ such that:
$$
Cl(E,q)=Cl^0(E,q)\oplus Cl^1(E,q)
$$
wich satisfy:
\begin{equation}
\begin{array}{ll}
Cl^0(E,q)\cdot Cl^0(E,q)=Cl^0(E,q),& Cl^1(E,q)\cdot Cl^1(E,q)=Cl^0(E,q)\\
Cl^0(E,q)\cdot Cl^1(E,q)=Cl^1(E,q),& Cl^1(E,q)\cdot Cl^0(E,q)=Cl^1(E,q)
\end{array}
\end{equation}
\end{proposition}

\begin{definition}
The group $\text{Spin}(E,q)$ is the subset of $Cl^0(E,q)$ defined by
$$
\{s\in Cl^0(E,q)|q(s)=1\}
$$
where $q$ is the extension of the quadratic form $q$ to $Cl(E,q)$.
\end{definition}
%

The formalism previously defined can of course be applied to the case of the Minkowski spacetime $(\R^4, \eta)$.
\begin{definition}
The bundle defined by:
$$
\ds_{Dirac}=\left(P_{\ds}\times M_2(\C)\right)/(Spin(1,3))
$$
is called the bundle of four dimensional spinors or bundle of Dirac spinors.
\end{definition}

\begin{remark}
\begin{enumerate}
\item The representation of $Spin(1,3)=SL_2(\C)$ acting on $M_2(\C)$ has two irreducible components, which correspond to $\C^2$ with its two inequivalent complex structures; by convention, we write:
$$
M_2(\C)= (\C^2)^{\ast}\oplus\overline{\C^2}.
$$ 
\item The previous remark gives a decompositon of the fibre bundle $\ds_{Dirac}$ into two bundles (known as bundles of Weyl spinors), corresponding to the splitting of $M_2(\C)$ into $\overline{\C}^2$ and $\left(\C^2\right)^\ast$; this decomposition is written in terms of indices:
$$
\ds_{Dirac}=\ds_A\oplus \ds^{A'}.
$$
A section $u$ of the $\ds_{Dirac}$ bundle will then be split into two smooth sections of the Weyl bundles:
$$
u=\phi_A\oplus\psi^{A'}.
$$
\item To get back to the tangent bundle, a convention must be chosen to represent the Clifford algebra $Cl(\R^4, \eta)$. Its  usual representation is $M_2(\mathbb{H})$, which is split in $M_2(\C)\oplus M_2(\C)$.  The vectors are identified with the hermitian 2-forms or with $\C$-antilinear homomorphisms from $\ds_A$ to $\ds^{A'}$. As such, a vector $u^a$ will be written by convention $u^{AA'}$.
\item For the chosen representation of the Clifford algebra which was made previously, the tangent bundle is identified to the set of  hermitian two forms over $\ds_A$. As such, it endows the tangent bundle with a structure of conformal Lorentzian manifold, i.e. a fibre bundle of cones over $M$ and a time orientation: the fibre bundle of cones over $\Omega$ is made of degenerate hermitian two-forms, the spacelike vectors fields are the hermitian matrices of signature (2,0) or (0,2) and the timelike vectors fields are the hermitian matrices of signature (1,1). The time orientation is obtained by a choice of orientation on the fibre bundle of cones over $\Omega$.

\end{enumerate}
\end{remark}

\begin{proposition}
The bundle $\Lambda^2 \ds^{A'}$ of skew-symmetric two forms is trivial.
\end{proposition}
\begin{proof}
This is a direct consequence of the fact that $\ds^{A'}$ is a $Spin(1,3)=SL_2(\C)$ bundle. Let $(U\times \overline{\C^2},\phi)$ and $(V\times\overline{\C^2}, \psi)$ be two local trivializations of the bundle $\ds ^{A'}$ with empty intersection where $\phi$ and $\psi$ satisfy
$$
p\circ\phi=\pi \text{ and } p\circ\psi=\pi
$$
where $\pi:\ds^{A'}\rightarrow M$ is the projection associated to the bundle $\ds^{A'}$ and $p$ is the projection on the first variable. These two trivializations give rise to two trivializations of $\Lambda^2\ds^A$ that are still denoted by $\phi$ and $\psi$. Let us consider the transition map $\phi\circ \psi^{-1}: \Lambda^{2} \C^2 \times V\rightarrow  \Lambda^{2} \C^2 \times U$. It can be written:
$$
\phi\circ\psi(x, y)=(x, \nu(x)y)
$$
where $\nu: U\cap V\rightarrow SL_2(\C)$ is a smooth map.\\ 
Let $x$ be fixed in $U\cap V$. Since $\nu(x)$ belongs to $SL_2(\C)$ and $y$ is a skew-symmetric 2-form, $y$ is invariant under the action of an element of $SL_2(\C)$, i.e:
$$
\forall (u,v)\in\C^2, y(\nu(x)u,\nu(x)v)= y(u,v).
$$
The fibre bundle $\Lambda^2\ds^{A'}$ is thus trivial.\fin 
\end{proof}

\begin{remark}
\begin{enumerate}
\item The canonical isomorphism, which will be denoted by $\kappa$, between  $\ds^{A'}$ and $\ds_{A}$ induces an other isomorphism between $\Lambda^2 \ds^{A'}$ and $\Lambda^2 \ds_{A}$:
$$
\begin{array}{ccl}
\Lambda^2 \ds^{A'} &\longrightarrow & \Lambda^2 \ds_{A}\\
\epsilon&\longmapsto&\kappa_\ast\epsilon:(u,v)\in\ds_{A}\times \ds_{A} \mapsto\epsilon(\kappa(u), \kappa(v)).
\end{array}
$$
It allows the construction of a symplectic form on $\ds_{\text{Dirac}}$: let $\epsilon$ in $\Lambda^2\ds^{A'}$, we obtain a symplectic form on $\ds_{Dirac}$ by taking:
$$
\epsilon\oplus \kappa^\ast\epsilon.
$$
A two-form $\eps$ on $\ds_A$ is denoted $\eps^{AB}$ and acts on Weyl spinors by:
$$
\forall (u_A,v_B)\in \ds_A, \eps(u,v)=\eps^{AB}u_Av_B.
$$
The corresponding two-form on $\ds^{A'}$ is denoted $\eps_{A'B'}$.

\item Let $\eps$ be a fixed skew-symmetric two-form on $\ds_A$. It is possible to construct a metric $\tilde g$ 
on $T\Omega$ by, for $x^{AA'}$ and $y^{AA'}$ two vectors:
$$
g_{ab}u^av^b=\eps_{AB}\eps_{A'B'}x^{AA'}y^{BB'}
$$
\item We denote by $\eps_{AB}$ a two form which gives rise to the metric $g$ on $M$. The non-degeneracy of $\eps$ induces an indentification between $\ds_A$ and its dual $\ds^A$ given by:
$$
\kappa_A\in \ds_A\longmapsto \kappa^A=\eps^{AB}\kappa_B \in \ds^A
$$
whose inverse mapping is
$$
\kappa^B\in \ds^B\longmapsto \kappa_B=\kappa^A\eps_{AB}.
$$ 
The equivalent transformation can be made for the complex conjugate spinors in $\ds^{A'}$ if we consider the image two-form $\eps_{A'B'}$.
\item The symplectic product on Dirac spinors can thus be written, by lowering and raising indices:
\begin{eqnarray*}
(u,v)&=&\eps^{AB}\psi_{A}\phi_{B}+\eps_{A'B'}\xi^{A'}\zeta^{B'}\\
&=&-\psi^{A}\phi_{A}+\xi_{A'}\zeta^{A'}
\end{eqnarray*}
where $u=\psi_{A}+\xi^{A'}$ and $v=\phi_{A}+\zeta^{A'}$ are two Dirac spinors.
\item The dual $\ds_{\text{Dirac}}^\star$ of $\ds_{\text{Dirac}}$ is split in:
\begin{equation*}
\ds_{\text{Dirac}}^\star=\ds_{A'}\oplus\ds^{A}.
\end{equation*} 
The symplectic form $(\cdot,\cdot)$ realizes an identification between $\ds_{\text{Dirac}}$ and $\ds_{\text{Dirac}}^\star$, whereas its restrictions to, respectively, $\ds_{A}$ and $\ds^{A'}$, denoted $\eps^{AB}$ and $\eps_{A'B'}$, realize an identification between $\ds^A$ and $\ds_{A}$ and between $\ds_{A'}$ and $\ds^{A'}$ respectively.
\end{enumerate}
\end{remark}
\begin{proposition}
Let $\eps$ be a section of $\Lambda^2 \ds^A$. Let $\tilde g$ be the metric associated with $\eps$.
Then, the metric $\tilde g$ is conformal to the metric $g$.
\end{proposition} 
\proof: Let $p$ in $M$. Let $X$ in $T_pM$ and $u,v$ in $\ds_A$. We assume that the vector $X$ is a light-like vector for the metric $g$, i.e.:
$$
g(X,X)=g_{ab}X^aX^b=0.
$$
A necessary and sufficient condition for $\tilde g_{ab}$ to be conformal to $g$ is that $\tilde g$ and $g$ have the same null cone structure, i.e. it is sufficient to show:
$$
\tilde g(X,X)=\tilde g_{ab}X^aX^b=0.
$$
Since $X^a$ is light-like vector, it can be written:
 $$
 X^a=u^A\overline u ^{A'}.
 $$ 
 if it is future directed and 
 $$
  X^a=-u^A\overline u ^{A'}
 $$
 if it is past directed. The calculation is performed for a future directed null vectors, but it is the same for a past directed one. 
Because of the skew-symmetry of $\eps_{AB}$, we have:
 $$
 \eps_{AB}u^Au^B=0
 $$
and then
$$
\tilde g_{ab}X^a X^b=\eps_{AB}u^Au^B\eps_{A'B'}\overline u^{A'}\overline u^{B'}=0.
$$
$X$ is thus still a null vector for $\tilde g$ and $\tilde g$ and $g$ are conformal metrics. \fin

\begin{remark}
The map:
\begin{equation*}
\begin{array}{ccc}
\Lambda^2 \ds^A&\longrightarrow & \{\phi g| \phi\in C^\infty(\Omega,\R^\ast_+)\}\\
\eps_{AB}&\longmapsto& g_{ab}=\eps_{AB}\eps_{A'B'}
\end{array}
\end{equation*}
is a two sheeted covering of the conformal class of $g$. In particular, it is surjective. We denote by  $\eps_{AB}$ a preimage of $g_{ab}$.
\end{remark}

\begin{proposition}
The bundle $\ds_{\text{Dirac}}$ is a Dirac bundle, i.e. a fibre bundle of left modules over $Cl(\Omega,g)$ endowed with a symplectic form $\epsilon$ and a connection $\nabla^{\ds}$ such that:
\begin{enumerate}
\item $\nabla^\ds$ is the pull-back of the Levi-Civita connection on $M$: if $\pi: \ds_{Dirac}\rightarrow \Omega$, then $\nabla$ can be written:
$$
\pi^\ast \nabla= \nabla^\ds
$$ 
\item the connection is compatible with the action of the Clifford algebra: let $X$ be a smooth section of $Cl(T\Omega,g)$ and $u$ a smooth section of $\ds_{Dirac}$, then:
$$
\nabla^\ds(X\cdot u)=\nabla X\cdot u+X\cdot \nabla^\ds u.
$$
Though different since they are acting on different objects, the connexion $\nabla$ on $\Omega$ and $\nabla^\ds$ on $\ds_{Dirac}$ are both denoted by $\nabla$.
\item the action of the Clifford multiplication is an isometry for the symplectic product: let $X$ be a smooth section of $Cl(T\Omega,g)$ and $u, v$ two smooth sections of $\ds_{Dirac}$, then:
$$
\epsilon(X\cdot u, X\cdot v)=q(X)\epsilon(u,v).
$$  
\end{enumerate}
\end{proposition}

In order to define the space on square integrable spinors on $\Omega$, it is necessary to define the norm of a  spinor. This unfortunately cannot be done without choosing a time function $t$ (see \cite{n02}) or, at least, a timelike vector field.

\begin{definition}
A smooth function $t$ on $\Omega$ is called a time function if, and only if its gradient is a non-vanishing future-oriented timelike vector field on $U$.
\end{definition}
\begin{definition}
Let $t$ be a time function on $\Omega$.
Then the map defined by:
$$
\begin{array}{ccc}
C^\infty_0(\Omega,\ds_{Dirac})\times C^\infty_0(\Omega,\ds_{Dirac})&\longrightarrow&\R\\
(\Psi, \Phi)&\longmapsto&\eps(\nabla t\cdot \Psi, \Phi)
\end{array}
$$
is a positive definitive hermitian product over the set of smooth sections of $\ds_{Dirac}$ with compact support in $U$. The norm associated to this scalar product is denoted by $||\star ||_{U}$.
\end{definition}
\begin{remark}
\begin{itemize}
\item This norm will be used in the following on compact subsets of an open set of $\Omega$ to define the Fréchet topology over smooth sections of the fiber bundle of Dirac spinors.
\item A time function $t$ is fixed on $\Omega$. This time function will be used to compute all the norms.
\end{itemize}
\end{remark}

This scalar product is used to define various norms over the spinor fields on $\Omega$: let $\Phi$ in $\mathcal{D}(\ds_{Dirac})$. We define using the positive definite hermitian product:
\begin{itemize}
\item the $L^\infty$-norm over a compact $K$ of $\Omega$: 
\begin{equation*}
||\Phi ||_{\infty, K}=\sqrt{\sup_{K} \left(\eps(\nabla t\cdot \Phi, \Phi)\right)};
\end{equation*}
\item if $\psi: \Omega \rightarrow \R^4$ is a given chart over $\Omega$, the norm over $K$, for any integer $N$:
\begin{equation*}
||\Phi ||_{\infty, N,K}=\sqrt{\sum _{\substack{\alpha\\ |\alpha|\leq N}}\sup_{K} \left(\eps(\nabla t\cdot \nabla^{\alpha}\Phi, \nabla^{\alpha}\Phi)\right)},
\end{equation*}
where $\nabla^\alpha=\nabla_{\partial_{\alpha_{1}}} \nabla_{\partial_{\alpha_{2}}}\dots\nabla_{\partial_{\alpha_{l}}}$ , $\alpha=(\alpha_{1},\dots,\alpha_{l} )$ being a multi-index of length\\
 $|\alpha|=\sum_{i=1\dots n}\alpha_{i}$; a chart of reference $\Psi$ is fixed in the following in the computation of the norms;
\item the $L^2$-norm over $\Omega$:
\begin{equation*}
||\Phi ||_{2}=\sqrt{\int_{\Omega}\eps(\nabla t\cdot \Phi, \Phi)\mu}.
\end{equation*} 
\end{itemize}

\subsubsection{Newman-Penrose tetrad.}

 One way to describe the Lorenztian structure is to use a global section of the fiber bundle of orthonormal frames over $\Omega$ and translate the result in terms of spinors. We construct a global basis, named tetrad of Newman-Penrose which gives rise to a spinor basis of $\ds^A$.

\begin{definition}
A basis of $T\Omega\otimes\C$ $(l,n,m,\overline{m} )$ is called a normalized Newman Penrose basis if $l$ and $n$ are real vectors fields and it satisfies the following relations:
\begin{equation*}
\begin{array}{rcrcrcr}
g(l,l)=0&,& g(n,n)=0&,& g(m,m)=0,&&\\
g(l,n)=1&,&g(m,\overline m)=-1&,&g(l,m)=0&,&g(n,m)=0.\\ 
\end{array}
\end{equation*}
\end{definition}
\begin{remark}

\begin{itemize}
\item The existence of a Newman-Penrose tetrad is insured by the existence of a global section of the fibre bundle of orthonormal frames: if $(e^a_{\mathbf{a}})$ $(\mathbf{a}=0,1,2,3)$ is such a section, the following family of vectors:
\begin{equation}
\begin{array}{cc}
\begin{array}{lcl}
l^{a}&= &\frac{1}{\sqrt2}(e_{0}^a+e_{1}^a)\\
n^{a}&= &\frac{1}{\sqrt2}(e_{0}^a-e_{1}^a)\\
\end{array}&
\begin{array}{lcl}
m^{a}&= &\frac{1}{\sqrt2}(e_{2}^a+ i e_{3}^a)\\
\overline m^{a}&= &\frac{1}{\sqrt2}(e_{2}^a- i e_{3}^a)\\
\end{array}
\end{array}
\end{equation} 
is a normalized Newman-Penrose tetrad. It is obvious that a given normalized Newman-Penrose tetrad gives rise to an orthonormal basis of $T\Omega$ with the following reverse fomulae:
 \begin{displaymath}
 \begin{array}{cc}
\begin{array}{lcl}
e^a_0&= &\frac{1}{\sqrt2}(l^a+n^a)\\
e^a_1&= &\frac{1}{\sqrt2}(l^a-n^a)\\
\end{array}&
\begin{array}{lcl}
e^a_2&= &\frac{1}{\sqrt2}(m^a+\overline{m}^a)\\
e^a_3&= &\frac{1}{i\sqrt2}(m^a-\overline{m}^a)\\
\end{array}
\end{array}
\end{displaymath}
\item Because the structure of null cones will be considered later, we assume that a Newman-Penrose tetrad  $(l,n,m,\overline m)$ is given first and, in a second time, gives rise to an orthonormal  basis $(e^a_{\mathbf{a}})$ $(\mathbf{a}=0,1,2,3)$.
\item Up to an overall sign, there exist two unique spinor fields in $\mathcal{E}(\ds^A)$, denoted by $o^A$ and $\iota^A$ such that:
\begin{equation*}
l^a=o^A\overline o^{A'}, n^a=\iota^A\overline \iota^{A'} \text{ and } m^a=o^A\overline \iota^{A'}.
\end{equation*}
These two spinors are chosen such that the following normalization is satisfied:
\begin{equation*}
\eps_{AB}o^A\iota^B=o_A\iota^A=1
\end{equation*}
\item There exists an alternative notation for this spin basis, which is consistent with the duality property used to describe spinors. We note, in $\ds^A$:
\begin{equation*}
\eps_0^A=o^A \text{ and } \eps_1^A=\iota^A.
\end{equation*}
We also introduce their dual spinors in $\ds_A$ $(\eps^0_A, \eps^1_A)$ which satisfy:
\begin{eqnarray*}
\eps_A^0\eps_0^A=1 &,& \eps_A^1\eps_1^A=1,\\
\eps_A^0\eps_1^A=0 &,& \eps_A^1\eps_0^A=0;
\end{eqnarray*}
they are:
\begin{equation*}
\eps^0_A=-\iota_A \text{ and } \eps^1_A=o_A.
\end{equation*}
\item The vector $e_{\mathbf{a}}^a$ can be written in function of the metric as $g_\mathbf{a}^a$ for $\mathbf{a}=0,\dots,3$. The components of its spinor form  $g_\mathbf{a}^{AA'}$, called the Infeld-van der Waerden, defined as:
\begin{equation*}
g_\mathbf{a}^\mathbf{AA'}=e^a_\mathbf{a}\eps_A^\mathbf{A}\eps_{A'}^\mathbf{A'}.
\end{equation*}
are the coefficients of the decomposition of $e_\textbf{a}^a$ in the basis $(\eps_0^A, \eps^A_1)$: 
\begin{equation*}
e^{a}_{\mathbf{a}}=g_\mathbf{a}^\mathbf{AA'}\eps^A_\mathbf{A}\eps^{A'}_\mathbf{A'}.
\end{equation*}
\end{itemize}
\end{remark}


It is then known (\cite{n02}, section 3.1) that the Clifford multiplication of a Dirac spinor by the basis vectors can be written:
\begin{lemma}
The Clifford multiplication of a Dirac spinor $\phi_A+\psi^{A'}$ by the vector $e^a_\mathbf{a}$ is given by:
\begin{equation*}
e^a_\mathbf{a}\cdot (\phi_A\oplus\psi^{A'})=i\sqrt{2}g^{\mathbf{a}}_{\phantom{a}AA'}\psi^{A'}\oplus-i\sqrt{2}g^{\mathbf{a}AA'}\phi_A
\end{equation*}
\end{lemma}
\begin{remark}\label{cliffcontr}: The Clifford multiplication can be interpreted as a contraction with the corresponding vector of the basis (up to a factor $\pm i\sqrt{2}$) by writing:
\begin{eqnarray*}
e^a_\mathbf{a}\cdot (\phi_A+\psi^{A'})&=&i\sqrt{2}g^{\mathbf{ab}}g_{\mathbf{b} AA'}\psi^{A'}-i\sqrt{2}g^{\mathbf{ab}}g_{\mathbf{b}}^{AA'}\phi_{A}\\
&=&i\sqrt{2}g(e_{\mathbf{a}}, e_{\mathbf{a}})g_{\mathbf{a}AA'}\psi^{A'}-i\sqrt{2}g(e_{\mathbf{a}}, e_{\mathbf{a}})g_{\mathbf{a}}^{AA'}\phi_{A}\\
\end{eqnarray*}
As a consequence, the Clifford multiplication by the vector $l^{AA'}$ is the contraction with $n^{AA'}$ and conversely the Clifford multiplication by $n^{AA'}$ is the contraction by $l^{AA'}$ (up to a factor $\pm i\sqrt{2}$):
\begin{equation}
\begin{array}{lcl}
l\cdot (\phi_A+\psi^{A'})&=&i\sqrt{2}(n_{AA'}\psi^{A'}-n^{AA'}\phi_A)\\
n\cdot (\phi_A+\psi^{A'})&=&i\sqrt{2}(l_{AA'}\psi^{A'}-l^{AA'}\phi_A)
\end{array}
\end{equation}
\end{remark}

We conclude this section by giving the abstract index expression of the Dirac operator on 4-spinors (\cite{n02}, section 3.1):
\begin{lemma}
The Dirac operator is decomposed as follows:
\begin{equation*}
\D (\phi_A+\psi^{A'})=i\sqrt{2}(\nabla_{AA'}\psi^{A'}-\nabla^{AA'}\phi_A)
\end{equation*}
\end{lemma}


%


%

%

\subsection{Analytic requirements}\label{analyticrequirements}

\subsubsection{Distributions on spinors} 

The purpose is to write weak solutions for the Dirac equation. The theory of distributions must thus be adapted to ensure  properties of symmetry for the Dirac operator and the Clifford multiplication so that the construction of Friedlander can be used with few adaptations.

\paragraph{Fundamental properties}

The basic elements needed in the next section are sketched here. 
 Spinor-valued distributions are defined in \cite{di82} to construct fundamental solutions for the Dirac equation. They were also developed in \cite{u97} to construct a Fourier integral operator for the propagator of the Dirac equation.

\begin{definition}
A distribution $u$ on the set $\mathcal{D}(\ds_{Dirac})$ of smooth Dirac spinor fields with compact support on $\Omega$, endowed with its usual Fréchet topology, is a $\C$-linear continuous mapping from $\mathcal{D}(\ds_{Dirac})$ to $\C$, i.e. a mapping which satisfies for all compact $K$ in $\Omega$, there exists a positive constant $C$ and an integer $m$ depending only on $K$ such that:
\begin{equation*}
\forall \phi \in \Dd, |u(\phi)|\leq C ||\phi ||_{\infty, m, K}
\end{equation*}
\end{definition}
The set of distributions on $M$ will be denoted by $\mathcal{D}'(\ds_{Dirac})$ and the duality bracket  by $<,>$.

\begin{definition}
The support of a distribution $u$ is the complement of the largest open subset $O$ of $\Omega$ such that any smooth function $\phi$ with support in $O$ satisfies:
\begin{equation*}
<u,\phi>=0.
\end{equation*}
\end{definition}
The set of compactly supported distributions is denoted $\mathcal{E}'(\ds_{Dirac})$ and is the topological dual of $\mathcal{E}(\ds_{Dirac})$, set of smooth sections of $\ds_{Dirac}$ on $\Omega$.

If $u$ is a locally integrable section of $\ds^{\ast}_{Dirac}=\ds^{A}\oplus \ds_{A'}$, which can be written $u=\xi^A+\eta_{A'}$, it defines a distribution by:
\begin{equation*}
\forall \Phi \in \mathcal{D}(\ds_{Dirac}), <u,\Phi>=\int_\Omega -\xi^A\phi_A + \eta_{A'}\psi^{A'}\mu. 
\end{equation*}
where the smooth section $\Phi$ is split as: $\Phi=\phi_A+\psi^{A'}$.

We define now the action of the covariant derivative in a direction $V$ and of the Dirac operator on distributions by:
\begin{proposition}\label{distdirac}
Let $u$ be an element of $\mathcal{D}'(\ds_{Dirac})$ and $V$ be a smooth section nowhere vanishing of $T\Omega$. The distributions $\nabla_V u$ and $\D u$ are defined by:
\begin{equation*}
\begin{array}{lccl}
\forall \phi \in\mathcal{D}(\ds_{Dirac}),& <\nabla_V u,\phi>_{\mathcal{D}'(\ds^\star_{\text{Dirac}}),\mathcal{D}(\ds_{\text{Dirac}})}&=&-<u,\nabla_V\phi>_{\mathcal{D}'(\ds^\star_{\text{Dirac}}),\mathcal{D}(\ds_{\text{Dirac}})}\\
\forall \phi \in \mathcal{D}(\ds_{Dirac}), & <\D u,\phi>_{\mathcal{D}'(\ds^\star_{\text{Dirac}}),\mathcal{D}(\ds_{\text{Dirac}})}&=&-<u,\D\phi>_{\mathcal{D}'(\ds^\star_{\text{Dirac}}),\mathcal{D}(\ds_{\text{Dirac}})}\\
\end{array}
\end{equation*}
\end{proposition}
These definitions agree with the Leibniz rule and the fact that the connexion is compatible with the symplectic product on spinors.

We also need to define the Clifford multiplication with a vector:
\begin{proposition}\label{distcliff}
Let $u$ be an element of $\mathcal{D}'(\ds_{Dirac})$ and $V$ a smooth section of $T\Omega$. We define the distribution $V\cdot u$ in  $\mathcal{D}'(\ds_{Dirac})$  by:
\begin{equation*}
\forall \phi \in \mathcal{D}(\ds_{Dirac}), <V\cdot u, \phi>_{\mathcal{D}'(\ds^\star_{\text{Dirac}}),\mathcal{D}(\ds_{\text{Dirac}})}=<u, V\cdot \phi>_{\mathcal{D}'(\ds^\star_{\text{Dirac}}),\mathcal{D}(\ds_{\text{Dirac}})}.
\end{equation*} 
\end{proposition}
\proof

The representation of the Clifford multiplication is the same for the dual $\ds_{Dirac}^\ast$.
Consequently, if $u=\phi^{A'}+\chi_{A}$ is in $\ds_{Dirac}$ and $v=\rho_{A'}+\theta^{A}$ is in $\ds_{Dirac}^\ast$, then:
\begin{equation*}
<v,e_\textbf{a}\cdot u>_{\ds^\star_{Dirac},\ds_{Dirac}}=-i\sqrt{2}g^{\textbf{a}}_{\phantom{a}AA'}\chi^{A'}\theta^{A}-i\sqrt{2}g^{\textbf{a}AA'}\phi_A\rho_{A'}. 
\end{equation*}
We notice that this expression is symmetric in $A$ and $A'$ so that we can conclude:
\begin{equation*}
<e_\textbf{a}\cdot v, u>_{\ds^\star_{Dirac},\ds_{Dirac}}= <v,e_\textbf{a} \cdot u>_{\ds^\star_{Dirac},\ds_{Dirac}}.\text{\fin}
\end{equation*}

\begin{remark} When a distribution on $\ds_{Dirac}$ is represented by a function from $\Omega$ into $\ds_{Dirac}$, the symplectic product $(\cdot,\cdot)$ on  $\ds_{Dirac}$ is used to apply the distribution on a section of $\ds_{Dirac}$. The duality bracket will be in that case written $(\cdot,\cdot)_{\mathcal{D}'(\ds_{Dirac}),\mathcal{D}(\ds_{Dirac})}$.
\end{remark}
 The previous results need to be checked since the definitions given in (\ref{distdirac})  do not work when the symplectic product (or the $\varepsilon$ spinor) is used. We first need the following lemmata on the action of Clifford multiplication and the Dirac operator:
\begin{lemma}
For any $\Phi$ and $\Psi$ Dirac spinor fields on $\Omega$ and $V$ a vector field on $\Omega$, we have:
\begin{equation*}
(V\cdot \Phi,\Psi)=-(\Phi,V\cdot\Psi)
\end{equation*}
\end{lemma}
\proof: It is sufficient to verify the result for an element $e_\textbf{a}$ of the frame. We calculate $(e_{\textbf{a}}\cdot \Phi,\Psi)$ in components. 
\begin{equation*}
(e_{\textbf{a}}\cdot \Phi,\Psi)=-i\sqrt{2}\varepsilon_{A'B'}g^{\textbf{a}AA'}\xi_{A}\theta^{B'}+i\sqrt{2}\varepsilon^{AB}g^{\textbf{a}}_{\phantom{a}AA'}\chi^{A'}\rho_{B}
\end{equation*}
with $\Phi=\xi_A+\chi^{A'}$ and $\Psi=\rho_{A}+\theta^{A'}$. Noticing that:
\begin{equation*}
\varepsilon^{AB}g^{\textbf{a}}_{\phantom{a}AA'}=-g^{\textbf{a}B}_{\phantom{BB}A'}=-\varepsilon_{B'A'}g^{\textbf{a}BB'}=\varepsilon_{A'B'}g^{\textbf{a}BB'}
\end{equation*}
we obtain:
\begin{equation*}
(e_{\textbf{a}}\cdot \Phi,\Psi)=-i\sqrt{2}\varepsilon^{AB}g^{\textbf{a}}_{\phantom{a}BB'}\xi_{A}\theta^{B'}+i\sqrt{2}\varepsilon_{A'B'}g^{\textbf{a}BB'}\chi^{A'}\rho_{B}=-(\Phi,e_{\textbf{a}}\cdot\Psi).\text{\fin}
\end{equation*}

In order to verify the symmetry of the Dirac operator for the symplectic product, we will establish the following lemma: 
\begin{lemma}\label{lemprau}
Let $\Phi$ and $\Psi$ two spinor fields on $\Omega$. Then we have:
\begin{equation*}
(\D\Phi,\Psi)=(\Phi,\D\Psi)-\textrm{div}(V).
\end{equation*} 
where $V$ is a complex vector field.
\end{lemma}
\proof: The formula is proved at each point of $\Omega$; let then $p$ be a point in $\Omega$. Let $(f_{i})$ be a orthonormal basis on $\Omega$ such that, for all $i$ in $\{0,1,2,3\}$:
\begin{equation*}
\nabla_{f_{i}}f_{i}=0 \text{ at } p. 
\end{equation*}

For this choice of basis, we have, at the point $p$:
\begin{eqnarray*}
(\D\Phi,\Psi)&=&\sum_{i\in \{0,1,2,3\}}(f_{i}\cdot \nabla_{f_i}\Phi,\Psi)\\
&=&-\sum_{i\in \{0,1,2,3\}}(\nabla_{f_i}\Phi,f_{i}\cdot \Psi)\\
&=&-\sum_{i\in \{0,1,2,3\}}\big\{\nabla_{f_i}(\Phi,f_{i}\cdot \Psi)-(\Phi,\nabla_{f_i}(f_{i}\cdot \Psi))\big\} \text{ $(,)$ being compatible with the connection.}\\
&=&-\sum_{i\in \{0,1,2,3\}}\big\{\nabla_{f_i}(\Phi,f_{i}\cdot \Psi)-(\Phi,f_{i}\cdot \nabla_{f_i}\Psi)\big\} \text{ since at $p$ } \nabla_{f_{i}}f_{i}=0\\
&=&(\Phi, \D\Psi)-\sum_{i\in \{0,1,2,3\}}\nabla_{f_i}(\Phi,f_{i}\cdot \Psi).
\end{eqnarray*}
Introducing the complex vector field $v$ defined, at $p$, by:
\begin{equation*}
V=\sum_{i=0}^{3}\mathfrak{f}_{i}(\Phi,f_{i} \cdot \Psi)f_i
\end{equation*}
with $\mathfrak{f}_i=<f_i,f_i>$, we notice that
$$
\sum_{i\in \{0,1,2,3\}}\nabla_{f_i}(\Phi,f_{i}\cdot \Psi)
$$
is the divergence of $V$.\\
We present an alternative way to perform this calculation with abstract indices; the Dirac spinors $\Phi$ and $\Psi$ are split on $\ds_{Dirac}$:
\begin{eqnarray*}
\Phi&=&\phi_{A}\oplus\rho^{A'}\\
\Psi&=&\psi_A\oplus\chi^{A'}.
\end{eqnarray*}
We now lead the calculation in the usual way:
\begin{eqnarray*}
\frac{1}{\sqrt{2}}(\D\Phi,\Psi)&=&\varepsilon^{AB}(i\nabla_{AA'}\rho^{A'})\psi_B+\varepsilon_{A'B'}(-i)(\nabla^{AA'}\phi_{A})\chi^{B'}\\
&=&\varepsilon^{AB}i\nabla_{AA'}(\rho^{A'}\psi_B)+\varepsilon_{A'B'}(-i\nabla^{AA'}(\phi_{A}\chi^{B'}))\\
&&-\varepsilon^{AB}\rho^{A'}i\nabla_{AA'}\psi_B-\varepsilon_{A'B'}\phi_{A}(-i\nabla^{AA'}\chi^{B'})\\
&=&i\nabla_{AA'}(\varepsilon^{AB}\rho^{A'}\psi_B)+(-i)\nabla^{AA'}(\varepsilon_{A'B'}\phi_{A}\chi^{B'})\\
&&+\rho^{A'}i\nabla_{A'}^B\psi_B-\phi_{A}(-i)\nabla^{A}_{B'}(\chi^{B'})\\ 
&=&i\nabla_{AA'}(\rho^{A'}\psi^A)+i\nabla^{AA'}(\phi_{A}\chi_{A'})\\
&&+\rho^{A'}i\nabla^{BB'}\varepsilon_{B'A'}\psi_B+i\phi_{A}\varepsilon^{AB}\nabla_{BB'}\chi^{B'}\\
&=&i\nabla_{AA'}(\rho^{A'}\psi^A)+i\nabla^{AA'}(\phi_{A}\chi_{A'})\\
&&-\varepsilon_{A'B'}\rho^{A'}i\nabla^{BB'}\psi_B+i\varepsilon^{AB}\phi_{A}\nabla_{BB'}\chi^{B'}\\
&=&i\nabla_{AA'}(\rho^{A'}\psi^A)+i\nabla^{AA'}(\phi_{A}\chi_{A'})+\frac{1}{\sqrt{2}}(\Phi,\D\Psi).
\end{eqnarray*}
It must be noticed that, in this new calculation, the remaining term can obviously be identified as a divergence.\fin
\begin{remark}\label{divergence}
The vector field
\begin{equation}\label{vectorfield}
V=\sum_{i=0}^{3}\mathfrak{f}_i (\Phi,f_{i} \cdot \Psi)e_i
\end{equation}
is encountered several times in the following. Though it is used to perform the calculation, it does not seem to be intrinsic. It is nonetheless easy to give a more intrinsic sense to this computation. Let us consider the complex 1-form $\omega$ on $\Omega$:
\begin{equation*}
\begin{array}{ccc}
T\Omega\otimes \C &\longrightarrow & \C\\
h&\longmapsto& (\Psi, h\cdot \Phi)
\end{array}
\end{equation*}
The dual vector of this 1-form is the vector \eqref{vectorfield}. The calculation can then be easily reinterpreted when noticing:
\begin{equation*}
\ud \star \omega =\left(\sum_{i\in \{0,1,2,3\}}\nabla_{f_i}(\Phi,f_{i}\cdot \Psi)\right)\mu,
\end{equation*}
$\star$ being the Hodge dual and $\mu$ the volume form associated with the metric $g$.
\end{remark}
\begin{definition}\label{clifforddiracdist}
Let $u$ be in $\mathcal{D}'(\ds_{Dirac})$, and $X$ in $C^\infty{T\Omega}$. The applications defined by 
\begin{equation*}
\Phi \in \mathcal{D}(\ds_{Dirac}) \longmapsto -(u,X\cdot\Phi)_{\mathcal{D}'(\ds_{\text{Dirac}}),\mathcal{D}(\ds_{\text{Dirac}})}
\end{equation*}
and 
\begin{equation*}
\Phi \in \mathcal{D}(\ds_{Dirac}) \longmapsto (u,\D\Phi)_{\mathcal{D}'(\ds_{\text{Dirac}}),\mathcal{D}(\ds_{\text{Dirac}})}
\end{equation*}
are distributions, denoted respectively by $X\cdot u$ and $\D u$.
\end{definition}
\proof: This is a straightforward consequence of the previous lemma and the Stokes theorem.\fin
\begin{remark}: These definitions agree with the previous lemmata when $u$ is in $\mathcal{D}(\ds_{Dirac})$.
\end{remark}

From this point, all the distributions will be assumed to be represented via the symplectic product.

If $f$ is in $\mathcal{D}'(\R)$ and $U$ is a smooth spinor field on $\Omega$, we define the distribution $fU$ by:
\begin{equation*}
\forall \phi\in \mathcal{D}(\ds_{Dirac}), (fU,\phi)_{\mathcal{D}'(\ds_{Dirac}), \mathcal{D}(\ds_{Dirac})}=<f,(U,\phi)>_{\mathcal{D}'(\R), \mathcal{D}(\R)}.
\end{equation*}

\paragraph{Composition of a function with a distribution}

In the following, the construction of distributions with support on a light cone will be required. One way to achieve this is to adapt the contruction of Friedlander in \cite{Friedlander:1975vn} in the case of spinor valued distribution.
\begin{definition}\label{compdis}
Let $S$ be a smooth function on $\Omega$, with non vanishing gradient on $\Omega$.\\
Let $f$ be a distribution with compact support on $\R$.\\
Then, the application
\begin{equation*}
\phi \in \mathcal{D}(\Omega)\longrightarrow \left(f(t),\int_{S(p)=t}\phi(p)\nabla S(p)\lrcorner \mu(p)\right)
\end{equation*}
where $\nabla S(p)\lrcorner \mu(p)$ is the contraction of the measure on $M$ with the gradient $\nabla S$ (or the Leray measure on the hypersurface $S(p)=t$), defines a real distribution denoted $f(S)$. This distribution coincides with the composition of functions when $f$ is represented by a function. 
\end{definition}
We need to apply this definition to calculate the action of the Dirac operator to a distribution on $\ds_{Dirac}$ of the form  $f(S)U$:
\begin{proposition}\label{derivation}
Let $f$ be an element of $\mathcal{E}'(\R)$, $S$ a smooth function chosen as in definition \ref{compdis} and $U$ a smooth spinor field on $M$.\\
Then, in the sense of distributions,
\begin{equation*}
\D\left(f(S)U\right)=f'(S)\hat\nabla(S)\cdot U+f(S)\D U,
\end{equation*}
where $\hat \nabla S$ is the raised gradient, i.e.
$$
\hat\nabla(S):=\sum_i \left(\nabla_{e_i}u\right) e_i.
$$ .
\end{proposition}
\proof: Let $\Phi \in \mathcal{D}(\ds_{Dirac})$ and $(f_i)$ an orthonormal frame. $\Phi$ is chosen with support in domain $\Omega$ where $\nabla _{f_i}f_i$ are all zero. We calculate $(\D\left(f(S)U\right),\Phi)_{\mathcal{D}'(\ds_{\text{Dirac}}),\mathcal{D}(\ds_{\text{Dirac}})}$ using the previous definitions and lemma \ref{lemprau}:
\begin{eqnarray}
(\D\left(f(S)U\right),\Phi)_{\mathcal{D}'(\ds_{\text{Dirac}}),\mathcal{D}(\ds_{\text{Dirac}})}&=&(f(S)U,\D\Phi)_{\mathcal{D}'(\ds_{\text{Dirac}}),\mathcal{D}(\ds_{\text{Dirac}})}\nonumber\\
&=&<f,\int_{S(p)=t}(U,\D\Phi)\mu_{S_t} >_{\mathcal{E}'(\R), \mathcal{E}(\R)}\nonumber\\
&=&<f,\int_{S(p)=t}(\D U,\Phi)\mu_{S_t} >_{\mathcal{E}'(\R), \mathcal{E}(\R)}\label{comp1}\\
&+&<f,\int_{S(p)=t}\nabla_{f_i}(U,f_i\cdot \Phi)\mu_{S_t} >_{\mathcal{E}'(\R), \mathcal{E}(\R)} \label{comp2}
\end{eqnarray}
where $\mu_{S_t}$ is the Leray measure $\nabla S\lrcorner \mu$ on the hypersurface $S_t=\{S(p)=t\}$. We calculate the two terms independently; by definition, \eqref{comp1} is:
\begin{equation*}
<f,\int_{S(p)=t}(\D U,\Phi)\mu_{S_t} >_{\mathcal{E}'(\R), \mathcal{E}(\R)}=(f(S)\D U,\Phi)_{\mathcal{D}'(\ds_{Dirac}), \mathcal{D}(\ds_{Dirac})}.
\end{equation*}
and (\ref{comp2}) is calculated using the same idea as in lemma (\ref{lemprau}):
\begin{eqnarray*}
<f,\int_{S(p)=t}\nabla_{f_i}(U,f_i\cdot \Phi)\mu_{S_t} >_{\mathcal{E}'(\R), \mathcal{E}(\R)}=<f,\int_{S(p)=t}\text{div}(v)\mu_{S_t} >_{\mathcal{E}'(\R), \mathcal{E}(\R)}
\end{eqnarray*}
where  $v$ is the vector field on $\Omega$ defined by:
\begin{equation*}
v=\sum_{i=0}^{3}\mathfrak{f}_i (U,f_i \cdot \Phi)f_i
\end{equation*}
with $\mathfrak{f}_i=<f_i,f_i>$. Noticing that:
\begin{equation*}
\frac{\ud}{\ud t}\int_{S(p)\leq t}\text{div}(v)\mu=\int_{S_t} \text{div}(v)\mu_{S_t}
\end{equation*}
and using the Stokes theorem
\begin{eqnarray*}
 \int_{S(p)\leq t} \text{div}(v)\mu&=&\int_{S_t}<\nabla S(p),v>\mu_{S_t}\\
 &=&\int_{S_t}\sum_i \nabla_{f_i}S(U,f_i\cdot \Phi)\mu_{S_t}\\
 &=&\int_{S_t}(U,\hat\nabla S\cdot \Phi )\mu_{S_t}\\
 &=&-\int_{S_t}(\hat\nabla S\cdot U,\Phi )\mu_{S_t},\\
 \end{eqnarray*}
 we obtain, accordingly with definition \ref{clifforddiracdist}:
 \begin{equation*}
 <f,\int_{S_t}(\nabla_{f_i}(U,f_i \cdot \Phi)\mu_{S_t} >_{\mathcal{E}'(\R), \mathcal{E}(\R)}=(f'(S)\hat\nabla S\cdot U,\Phi)_{\mathcal{D}'(\ds_{\text{Dirac}}),\mathcal{D}(\ds_{\text{Dirac}})}
 \end{equation*}
 so that:
 \begin{equation*}
 \D\left(f(S)U\right)=f'(S)\hat\nabla(S)\cdot U+f(S)\D U\text{ \fin }
 \end{equation*}

\paragraph{Spinors and bidistributions}

Keeping in sight that the purpose is to write an integral formula (or representation formula) for a Cauchy problem, we must be able to apply twice a distribution to spinor fields. This is what bidistributions are made for.

We define the product $\boxtimes$ of two smooth sections of $\ds_{Dirac}$ with compact support by:
\begin{equation*}
\begin{array}{ccc}
\mathcal{D}(\ds_{Dirac})\times \mathcal{D}(\ds_{Dirac})&\longrightarrow &\mathcal{D}(\ds_{Dirac})\boxtimes \mathcal{D}(\ds_{Dirac})\\
(\Phi, \Psi)&\longmapsto& ((p,q)\in \Omega\times \Omega \mapsto \Psi(p)\otimes \Phi(q))
\end{array}
\end{equation*}
where $\otimes$ must be understood as tensor product of spinors in different variables. The vector space generated by these products is denoted by  $\mathcal{D}(\ds_{Dirac})\boxtimes \mathcal{D}(\ds_{Dirac})$.

\begin{definition}
Let $u$ and $v$ be two distributions in $\mathcal{D}'(\ds_{Dirac})$.
The bidistribution $u\boxtimes v$ is an application from $\mathcal{D}(\ds_{Dirac})\boxtimes \mathcal{D}(\ds_{Dirac})$ defined by, for every $(\phi, \psi)\in \mathcal{D}(\ds_{Dirac})\times \mathcal{D}(\ds_{Dirac})$:
\begin{equation*}
(u\boxtimes v, \phi\boxtimes \psi)=(u, \phi)_{\mathcal{D}(\ds_{Dirac}), \mathcal{D}'(\ds_{Dirac})}(v, \psi)_{\mathcal{D}'(\ds_{Dirac}), \mathcal{D}(\ds_{Dirac})}.
\end{equation*}
The vector space $\mathcal{D}'(\ds_{Dirac}) \boxtimes \mathcal{D}'(\ds_{Dirac})$ generated by these products is called the space of spinor-valued  bidistributions on $\ds_{Dirac}$.
\end{definition}
If $\phi$ is in $\mathcal{D}(\ds_{Dirac})$ and $u$ is a spinor bidistribution, then $u(\phi)$ is still in $\mathcal{D}'(\ds_{Dirac})$. It can consequently be still applied to a function in $\mathcal{D}(\ds_{Dirac})$.


A special type of spinor valued distribution that will be encountered in the following is the Dirac distribution.
\begin{definition}
We define the Dirac distribution (or Dirac mass) in $p'$, denoted by $\overline{\delta}_{p'}$ by:
\begin{equation*}
\forall \phi\in \Dd, (\overline\delta_{p'},\phi)_{\mathcal{D}'(\ds_{\text{Dirac}}),\mathcal{D}(\ds_{\text{Dirac}})}=\phi(p').
\end{equation*}
\end{definition}
It must be noted that this distribution can be written in the form $\tau(p',p)\delta_{p'}$ (\cite{Friedlander:1975vn}, chapter 6) where $\tau$ is a linear transformation from $\mathcal{D}(\ds_{Dirac})$ in the variable $p$ to  $\mathcal{D}(\ds_{Dirac})$ in the variable $q$ satisfying $\tau(p',p')=I_{\ds_{Dirac}}$ and can consequently be written as:
\begin{equation*}
\forall \phi \in \Dd, (\tau(p',p)\delta_{p'},\phi)_{p,\{\mathcal{D}'(\ds_{Dirac}), \mathcal{D}(\ds_{Dirac})\}}=\phi(p'),
\end{equation*}
the duality bracket being computed in the variable $p$.
Since 
\begin{equation*}
\varepsilon^{AB}\varepsilon^{\,0}_{A}\varepsilon^{\,1}_{B}=1 \text{ and } \varepsilon_{A'B'}\varepsilon^{\,A'}_{0'}\varepsilon^{\,B'}_{1'}=-1
\end{equation*} 
and $\tau(p,p)$ satisfies:
\begin{equation*} 
(\tau(p,p),\phi(p))=\phi(p) 
\end{equation*}
it can be explicitly calculated at $p=p'$:
\begin{eqnarray} 
\tau(p,p)&=&-\varepsilon^{\,A'}_{1'}\boxtimes\varepsilon^{\,A'}_{0'}+\varepsilon^{\,A'}_{0'}\boxtimes\varepsilon^{\,A'}_{1'}+\varepsilon^{\,0}_{A}\boxtimes\varepsilon^{\,1}_{A}-\varepsilon^{\,1}_{A}\boxtimes\varepsilon^{\,0}_{A\nonumber}\\
&=&-\overline{\iota}^{B'}\boxtimes\overline{o}^{A'}+\overline{o}^{B'}\boxtimes \overline{\iota}^{A'}-o_{B}\boxtimes\iota_{A}+\iota_{B}\boxtimes o_{A}\label{tau}.
\end{eqnarray}
Such a function $\tau$ is chosen explicitly later (see equation \eqref{functiontau}).

\subsubsection{Fundamental solutions of the wave equation}
We now apply  to the spinorial wave equation the analytical tools used by Friedlander in \cite{Friedlander:1975vn} for the tensor wave equation. An alternative method has been used by Klainerman and Rodnianski to construct an approximate fundamental solution in \cite{MR2339803}. Though their method is more flexible and well-suited to obtain estimates, it is not appropriate here since, as we will see, the regular part (the tail of the fundamental solution) is needed to write down a fundamental solution. V. Moncrief used Friedlander's method in a paper with D. Eardley (\cite{MR649158}) for the Yang-Mills equations in the Minkowski space and for the Maxwell wave equation in \cite{mon06} on a curved space-time.

We first consider the spinorial wave operator $\D^2$. The Schrödinger - Lichnerowicz - Böchner formula gives that for any $\phi$ in $\mathcal{D}(\ds_{Dirac})$:
\begin{equation}\label{sch}
\D^2\phi=\square\phi+\frac{1}{4}\text{Scal}\phi
\end{equation}
where $\square=-\nabla_j\nabla^j$. Since the index notations are used from the beginning, a index version of the formula with its proof is given:
\begin{proposition}[Schrödinger-Lichnerowicz formula in index version for spin $\frac12$]\label{slf12}
$\phantom{l}$\\Let $\phi_{A}$ be a smooth section of $\ds_{A}$.
Then we have the following relation:
\begin{equation*}
\nabla_{BA'}\nabla^{AA'}\phi_{A}=\frac{1}{2}\left(\nabla_{CC'}\nabla^{CC'}\phi_{B}+\frac{1}{4}\text{Scal}_g \phi_{B}\right).
\end{equation*}
\end{proposition}
\proof: The reader should refer for intermediate results to \cite{Penrose:1986fk}(4.9.2 and 4.9.17).
\begin{eqnarray*}
\nabla_{BA'}\nabla^{AA'}\phi_{A}&=&\eps^{AC}\nabla_{BA'}\nabla^{A'}_C\phi_{A}\\
&=&\eps^{AC}\left(\nabla_{[B|A'}\nabla^{A'}_{|C]}\phi_{A}+\nabla_{(B|A'}\nabla^{A'}_{|C)}\phi_{A}\right)\\
&=&\frac{1}{2}\eps^{AC}\nabla_{HH'}\nabla^{HH'}\eps_{BC}\phi_{A}+\frac{1}{8}\text{Scal}_g \phi_B\text{ (formula 4.9.17 in \cite{Penrose:1986fk})  }\\
&=&\frac{1}{2}\nabla_{CC'}\nabla^{CC'}\phi_{B}+\frac{1}{8}\text{Scal}_g\phi_B\text{\fin}
\end{eqnarray*}
\begin{remark}
\begin{itemize}
\item This version agrees with the previous one when noticing that the operator $\nabla_{BA'}\nabla^{AA'}$ is in fact, due to the renormalization induced by the Clifford multiplication, the projection on $\ds_B$ of $1/2 \D^2$.
\item A generalization of this formula to arbitrary spin is given later in subsection \ref{generalization}.
\end{itemize}
\end{remark}


Since $\Omega$ is a geodesically convex domain, it is possible to define globally on $\Omega$ the squared-distance function:
\begin{equation*}
\Gamma_{p}(q)=\int_{0}^t g\left(\frac{\ud \gamma(s)}{\ud s}, \frac{\ud \gamma(s)}{\ud s}\right)\ud s
\end{equation*}
where $\gamma: [0,t]\rightarrow \Omega$ is the unique geodesic from $p$ to $q$.

To write the fundamental solutions of the wave equation, it is necessary to construct distributions with support on a cone: using definition \ref{compdis}, let us consider the distributions
$$
\delta^\pm(\Gamma_{p}(q)) \text{ and } H^\pm(\Gamma_{p}(q))
$$
where $\delta$ is the Dirac mass and $H$ the Heaviside function. These distributions have support respectively, for $p$ fixed in $\Omega$, in $C^{\pm}(p)$ and $\mathcal{J}^\pm(p)$.
\begin{remark}\label{derivdistrib} It is important to notice that these distributions do not satisfy definition $\ref{compdis}$ since the gradient of $\Gamma_{p}(q)$ vanishes at the vertex of the cone. Nonetheless, considering the distributions
$$
\delta^\pm(\Gamma_{p}(q)-\eps) \text{ and } H^\pm(\Gamma_{p}(q)-\eps)
$$
 with $\eps$ positive avoids the problem. The results can then be obtained using a limiting process when $\eps$ tends to zero. This method will be used later to expand equation $\eqref{justif1}$.
\end{remark}
It is known that the operator $\D^2$ admits fundamental solutions (\cite{Friedlander:1975vn},\cite{di82}):
\begin{theorem}
There exists two bidistributions on $\Omega$, $\tilde G^\pm_q(p)$ that satisfy:
\begin{equation*}
\forall (p,q) \in \Omega^2, \D_p^2 \tilde G^\pm_q(p)=\overline \delta_q(p)
\end{equation*}
in the distribution sense. These two bidistributions can be written:
\begin{equation*}
\tilde G^\pm_q(p)= \tilde U_q(p)\delta^\pm(\Gamma_q(p))+\tilde V_q(p)H^\pm(\Gamma_q(p)).
\end{equation*}
where  $\tilde U$ and $\tilde V^\pm$ are smooth functions of the variable $(p,q)$. $q$ being fixed in $\Omega$, the support of $\tilde G_{q}^\pm(p)$ is then in $\mathcal{C}^\pm(q)$.
\end{theorem}

The structure of the fundamental solution obtained by Friedlander is the following (the reader should refer to \cite{Friedlander:1975vn} for more details.)
\begin{enumerate}
\item The function $\tilde U$ in the singular part can be decomposed into two parts, $\tilde U_q(p)=k_q(p)\tilde \tau_q(p)$ where:
\begin{enumerate}
\item the bispinor $\tilde \tau_q(p)$ satisfies:
\begin{equation}\label{functiontau}
\nabla^i \Gamma_q(p)\nabla_i \tilde\tau_q(p)=0 \text{ and } \tilde \tau_p(p)= \tau_p(p).
\end{equation}
This equation can easily be reinterpreted as parallel transport in the variable $q$ of the bispinor identity along the geodesic from $p$ to $q$.
\item the function $k_q(p)$ satisfies the transport equation:
\begin{equation}
2<\nabla \Gamma_q(p), \nabla k_q(p)>+(\square \Gamma_q(p)-8)k_q(p)=0 \text{ and } k_p(p)=\frac{1}{2\pi}.
\end{equation}
$k_q(p)$ measures the difference between the measure induced on $\mathcal{C}^+(p)\cap \mathcal{C}^-(q)$ and the measure on the standard sphere $S^2$ in the sense that, if $p$ is in the future of $q$:
$$
\mu_{\mathcal{C}^+(q)\cap \mathcal{C}^-(p)}=k_q(p) r^2 \mu_{S^2}
$$
where $\mu_{\mathcal{C}^+(q)\cap \mathcal{C}^-(qp)}$ is the Riemannian volume form induced by the metric $g$ on $\mathcal{C}^+(q)\cap \mathcal{C}^-(p)$ and $\mu_{S^2}$ the standard volume form on the two dimensional sphere. 
\end{enumerate}
\item The regular part $\tilde V^\pm$ of the fundamental solution can be obtained by solving the characteristic Cauchy problem:
\begin{equation*}
\left\{
\begin{array}{lcl}
\square \tilde V_q(p) &=&0 \text{ for }p \in \mathcal{J}^+(q)\\
\tilde V_q(p)&=&\tilde V^0_q(p)\text{ for }p \in \mathcal{C}^+(q)
\end{array}
\right.
\end{equation*}
where $\tilde V^0_q(p)$ satisfies the transport equation:
\begin{equation*}
2<\nabla \Gamma_q(p), \nabla \tilde V^0_q(p)>+(\square \Gamma_q(p)-4)\tilde V^0_q(p)=-D^2\tilde U.
\end{equation*}
\end{enumerate}

For later convenience, the fundamental solution must be split over the decomposition of the Dirac spinors:
\begin{equation*}
\tilde G^\pm_q(p)=\gu^\pm_{\underset{A}{q}}\underset{B}{(p)}+\gd^\pm_{\overset{A'}{q}}\overset{B'}{(p)}
\end{equation*}
The notation ${\overset{q}{A}}$ means that the part of the bidistribution in the variable $q$ acts on spinor fields in $\ds_A$. Their fundamental part is denoted by, respectively, $\Uu^\pm_{\underset{A}{q}}\underset{B}{(p)}$ and $\Ud^\pm_{\overset{A'}{q}}\overset{B'}{(p)}$.

Two backward and forward fundamental solutions for the wave equation can then be constructed. For Dirac spinors, these fundamental solutions are the distributions:
\begin{equation*}
D^p \tilde G^\pm_q(p)
\end{equation*}
In terms of indices, these fundamental solutions are written:
\begin{equation*}
\nabla_p^{BB'} \gu^\pm_{\underset{A}{q}}\underset{B}{(p)} \text{ on } \ds_ A\boxtimes \ds^{B'}
\text{ and }
\nabla^p_{BB'} \gd^\pm_{\overset{A'}{q}}\overset{B'}{(p)} \text{ on } \ds^{A'}\boxtimes \ds_{B}.
\end{equation*}

Finally, we state the following theorem concerning the existence and the structure of the fundamental solution for the Dirac equation for Dirac spinors.
\begin{theorem}
There exist two fundamental solutions for the Dirac operator $\D$, $G^\pm_q(p)$, with support in $\mathcal{C}^\pm(q)$, for $q$ fixed in $\Omega$,  such that:
\begin{equation*}
\forall (p,q) \in \Omega^2, \D^p G^\pm_q(p)=\overline \delta_q(p)
\end{equation*}
in the distribution sense. These two fundamental solutions are obtained by applying the Dirac operator to the two fundamental solutions of the wave equation:
\begin{equation*}
G^\pm_q(p)=\D^p \tilde G^\pm_q(p).
\end{equation*}
\end{theorem}

\section{Derivation of the integral formula for Dirac spinors}

This section is devoted to the derivation of an integral formula for Dirac spinors for the characteristic Cauchy problem with data on a future null cone. In this context, we will work with the forward fundamental solution $G^+_q(p)$ which will be denoted with no ambiguity $G_q(p)$. The singular and smooth parts of the forward fundamental solution for the wave equation will be denoted $\tilde U_q(p)$ and $\tilde V_q(p)$.

The point $p_0$ being fixed,  let $p$ be a point in the future of $p_{0}$ in $\Omega$. We define, for these two points:
\begin{itemize} 
\item $\sigma(p)=\mathcal{C}^+(p_0)\cap \mathcal{C}^{-}(p)$
\item $\mathcal{D}(p)=\mathcal{C}^+(p_0)\cap \mathcal{J}^-(p)$
\item $\mathcal{S}(p)=\mathcal{J}^+(p_0)\cap \mathcal{C}^{-}(p)$.
\item $\mathcal{V}(p)=\mathcal{J}^+(p_0)\cap \mathcal{J}^{-}(p)$
\end{itemize}
Since $\Omega$ is geodesically convex, these instersections are well-defined (in fact, the hypothesis of global hyperbolicity suffices).

\subsection{Representation formula}

 The first step to obtain a representation formula is to solve the problem with source:
\begin{equation*}
\D u= f.
\end{equation*}
The following lemma is a direct transcription of lemma 5.5.1 in \cite{Friedlander:1975vn}:
\begin{lemma}\label{lemme551}
Let $f$ in $\mathcal{E}(\ds_{Dirac})$.\\
Then the distributions defined by:
$$
\forall \phi \in \mathcal{D}(\ds_{Dirac}), (u, \phi)_p:=(f,(G^\pm_p,\phi)_q)_p
$$
are solutions of the problem:
$$
\D u= f.
$$
\end{lemma}
\proof: The calculation is made first formally. The justication of each step will be carried out later; it will be sufficient to check that each duality bracket is well-defined and that all the operations involved (symmetry on Dirac operator, \dots) are legitimate.\\
Let $\phi$ be in $\mathcal{D}(\ds_{Dirac})$.
\begin{eqnarray}
(\D^p u,\phi)_{p}&=&(u, \D^p \phi)_p\label{lemme5511}\\
&=&(f,(G^\pm_p,\D^q \phi)_q)_p \text{ by definition of $u$}\label{lemme5512}\\
&=&(f,\phi) \text{ by definition of $G^\pm_p$}\label{lemme5513}.
\end{eqnarray}

It must be checked to insure that \eqref{lemme5512} exists that the function:
\begin{equation*}
p\longmapsto (G^\pm_p,\D^q \phi)_q
\end{equation*} 
is smooth; we have:
\begin{equation}\label{55115}
(\D^q \tilde G^\pm_p, \D^q \phi)_q=(\tilde G^\pm_p, (\D^q)^2\phi)_q=\int_{\mathcal{C}^+(p)}(\tilde U^\pm_p(q),(\D^q)^2\phi )\mu_{\Gamma_p(q)}(q)+\int_{\mathcal{J}^+(p)}(\tilde V^\pm_p(q),(\D^q)^2\phi )\mu(q),
\end{equation} 
where $\mu_{\Gamma_{p}(q)}$ is the Leray form associated with the function $\Gamma_{p}(q)$, i.e:
$$
\mu_{\Gamma_{p}(q)}=\nabla^q\Gamma_{p}(q)\lrcorner \mu.
$$
 Let $\pi: \Omega \rightarrow \R^4$ be a chart recovering $\Omega$ (which exists since $\Omega$ is geodesically convex). The image by $\pi$ of $p$ and $q$  are respectively denoted by $y$ and $x$. There exists a diffeomophism $\xi\rightarrow x=h(y,\xi)$ from $\pi(\Omega)$ into $\R^4$, where $\xi=(\xi^0,\xi^1, \xi^2, \xi^3)$ is a coordinate system centered at $y$, Minkowskian in $q$ and such that the vector $(1,0,0,0)$ is timelike and future oriented. In this coordinate system, the measures $\mu$ and $\mu_{\Gamma_q(p)}$ are expressed as:
$$
\mu(q)=k(y,\xi)\ud \xi \text{ and } \mu_{\Gamma_q(p)}=k(y,\xi)\frac{\ud \xi^1 \wedge \ud \xi^2 \wedge \ud \xi^3}{2\sqrt{(\xi^1)^2+(\xi^2)^2+(\xi^3)}}
$$
with $\ud \xi=  \ud \xi^0\wedge\ud \xi^1 \wedge \ud \xi^2 \wedge \ud \xi^3$ and 
$$
C^+(q)=\big\{\xi| \xi^0=\sqrt{(\xi^1)^2+(\xi^2)^2+(\xi^3)^2}\big\} \text{ and } \mathcal{J}^+(q)=\big\{\xi| \xi^0\geq\sqrt{(\xi^1)^2+(\xi^2)^2+(\xi^3)^2}\big\}.
$$
The integral \eqref{55115} can then be rewritten:
\begin{gather}
(\D^q \tilde  G^\pm_p, \D^q \phi)_q=\int_{\xi^0=\sqrt{(\xi^1)^2+(\xi^2)^2+(\xi^3)^2}}(\tilde U^\pm_{h(y,\xi)}(y),\big((\D^q)^2\phi\big) (h(y,\xi)))k(y,\xi)\frac{\ud \xi^1 \wedge \ud \xi^2 \wedge \ud \xi^3}{2\sqrt{(\xi^1)^2+(\xi^2)^2+(\xi^3)}}\\
+\int_{\xi^0\geq \sqrt{(\xi^1)^2+(\xi^2)^2+(\xi^3)^2}}(\tilde V^\pm_{h(y,\xi)}(y),\big((\D^q)^2\phi\big) (h(y,\xi))k(y,\xi)\ud \xi.
\end{gather}
which is clearly a smooth function of $y=\pi(p)$.\\
Since $f$ is a distribution with compact support, there exists $K'$, an integer $N$ and a positive constant $C$ such that, for any smooth function $\psi$, the following estimate holds:
$$
|(f,\psi)|\leq C\sum_{|\alpha|<N}\sup_{y\in \pi(K')} ||\partial^\alpha_y \psi \circ \pi^{-1}(y)||.
$$ 
Let $K$ be a compact of $\Omega$. Assume that $\phi$ has its support in $K$. Then the previous inequality gives for $\psi =(G_p,\D \phi)_q$
$$
|(f,\psi)|\leq C'\sum_{|\alpha|<N}\sup_{y\in \pi(K')} ||\partial^\alpha_y (G_{\pi^{-1}(y)}(\pi^{-1}(x)),\phi)_x ||.
$$ 
Using the expression of $(G_{\pi^{-1}(y)}({\pi^{-1}(x)}),\phi)_x$, its derivatives $\partial^\alpha_y (G_{\pi^{-1}(y)}({\pi^{-1}(x)}),\phi)_x$ are bounded by the derivatives of $\phi$ on $K$:
$$
\sup_{K} ||\partial^\alpha_y (G_{\pi^{-1}(y)}(\pi^{-1}(x)),\phi)_x||\leq C_{K,K',\alpha}\sum_{|\beta\leq |\alpha|+1} \sup_{y\in \pi(K)} ||\partial^\beta_y \phi \circ \pi^{-1}(y)||
$$
where the constant $C_{K,K', \alpha}$ is determined only by the derivatives of $\tilde U$, $\tilde V$, $h$ of order up to $k+1$ on the compact $K\times K'$ and its image by $\pi$.
Finally, we obtain:
$$
|(f,\psi)|\leq \sum_{|\alpha|<N+1}\sup_{y\in \pi(K')} ||\partial^\alpha_y( \phi\circ\pi^{-1})  ||,
$$
which means that equation \eqref{lemme5512} is well-defined.\fin

Let $u$ be in $\mathcal{E}(\ds_{Dirac})$. The following proposition gives a representation of $u$ in term of its data on a null cone:
\begin{theorem}\label{representationformula} Let $u$ be a function with future bounded support.
Let $p_0$ in $\Omega$. Then, we have in the distributional sense:
\begin{gather*}
uH^+(\Gamma_{0})=
\D^q\left(\left(\int_{\mathcal{S}(q)} \big(\D^pu,\tilde U_{p}(q)\big) \mu_{\Gamma_{q}}(p)+\int_{\mathcal{V}(q)} \big(\D^pu,\tilde V_{p}(q)\big) \mu_{\Gamma_{q}}(p)\right.\right.\\
\left.\left.+\int_{\sigma(q)} \big(\nabla^p \Gamma_{0}\cdot u,\tilde U_{p}(q)\big) \mu_{\Gamma_{0},\Gamma_{q}}(p)+\int_{\mathcal{D}(q)} \big(\nabla^p \Gamma_{0}\cdot u,\tilde V_{p}(q)\big)  \mu_{\Gamma_{0}}(p)\right)H^+(\Gamma_{0})\right),
\end{gather*}
where the two-form  $\mu_{\Gamma_{0},\Gamma_{q}}(p)$ is obtained via the factorization:
\begin{equation*}
\forall \phi \in C^{\infty}_0(\Omega\times \Omega), \int_{A}\phi \mu_{\Gamma_0}(p)\wedge \mu_{\Gamma}(q)=\int_{\mathcal{J}^{+}(p_0)}\int_{\sigma(q)}\phi\mu_{\Gamma_0,\Gamma_q}(p)\wedge \mu(q),
\end{equation*}
where $A$ is the set $\{(p,q)|p\in \mathcal{C}^+(p_{0}) \text{ and } q \in \mathcal{C}^+(p) \}$. 
\end{theorem}
\proof
Let $u$ be a function with future bounded support, that is to say that the intersection of $supp(u)$ with any future null cone is compact . 
We use here the property of the fundamental solution with lemma \ref{lemme551} with $f=\D^p(uH^\pm(\Gamma_0))$: 
\begin{eqnarray}
(uH^+(\Gamma_0),\phi)_{p}&=&\left( \D^p\left(uH^+(\Gamma_0)\right), \left(G_{p},\phi\right)_{q}\right)_{p}\label{justif2}\\
&=&\left( \D^p\left(uH^+(\Gamma_0)\right), \left(\D^q\tilde G_{p},\phi\right)_{q}\right)_{p}\nonumber\\
&=&\left( \D^p\left(uH^+(\Gamma_0)\right), \left(\tilde G_{p},\D^q\phi\right)_{q}\right)_{p}\label{justif1}
\end{eqnarray}
The duality bracket \eqref{justif1} is properly defined since the function $p\mapsto (\tilde G_{p},\D^q\phi)_{q}$ is a smooth function with support in the future of $\text{Supp}(\phi)$, that is to say $\cup_{q\in \text{Supp}(\phi)}\mathcal{I}^+(q)$, and since $u$ has future bounded support.

The duality bracket \eqref{justif1} is then developed. The first step consists in differentiating the distributions $uH^+(\Gamma_{0})$. As already noticed in remark \ref{derivdistrib}, the distribution $uH^+(\Gamma_{0})$ is not of the type given in proposition \ref{derivation} since $\nabla \Gamma_{0}$ vanishes at $p_{0}$. To avoid this difficulty, we consider the distributions $u H^+(\Gamma_{0}-\eps)$, where $\eps$ is a positive number. This derivation gives, since proposition \ref{derivation} can be applied:
$$
\D\left(uH^+(\Gamma_{0}-\eps)\right)=(\D u) H^+(\Gamma_{0}-\eps)+\hat \nabla \Gamma_{0} \cdot u \delta^+(\Gamma_{0}-\eps)
$$   
which becomes, when $\eps$ tends to zero:
$$
\D\left(uH^+(\Gamma_{0})\right)=(\D u) H^+(\Gamma_{0})+\hat \nabla \Gamma_{0} \cdot u \delta^+(\Gamma_{0}).
$$
The bracket \eqref{justif1} is written as the sum of four integrals:
\begin{eqnarray*}
\left( \D^p\left(uH^+(\Gamma_0)\right), \left(\tilde G_{p},\D^q\phi\right)_{q}\right)_{p}&=&\int_{\mathcal{J}^+(p_{0})}\int_{\mathcal{C}^+(p)} (\D^pu,(\tilde U_{p}(q), \D^q\phi)) \mu_{\Gamma_{p}}(q)\wedge \mu(p) \\
&+&\int_{\mathcal{J}^+(p_{0})}\int_{\mathcal{J}^+(p)} (\D^pu,(\tilde V_{p}(q), \D^q\phi)) \mu(q)\wedge \mu(p)\\
&+&\int_{\mathcal{C}^+(p_{0})}\int_{\mathcal{C}^+(p)} (\nabla^p \Gamma_{0}\cdot u,(\tilde U_{p}(q), \D^q\phi)) \mu_{\Gamma_{p}}(q)\wedge \mu_{\Gamma_{0}}(p)\\
&+&\int_{\mathcal{C}^+(p_{0})}\int_{\mathcal{J}^+(p)} (\nabla^p \Gamma_{0}\cdot u,(\tilde V_{p}(q), \D^q\phi)) \mu(q)\wedge \mu_{\Gamma_{0}}(p),
\end{eqnarray*}
where $\mu_{\Gamma_{0}}$ and $\mu_{\Gamma_{p}}$ are the Leray measures associated with $\Gamma_{0}$ and $\Gamma_{p}$ respectively.

Switching the order of integration of the variables, we get:
\begin{eqnarray*}
\left( \D^p\left(uH^+(\Gamma_0)\right), \left(\tilde G_{p},\D^q\phi\right)_{q}\right)_{p}&=&\int_{\mathcal{J}^+(p_{0})}\int_{\mathcal{C}^-(q)\cap \mathcal{J}^+(p_{0})} \big((\D^pu,\tilde U_{p}(q)), \D^q\phi\big) \mu_{\Gamma_{q}}(p)\wedge \mu(q) \\
&+&\int_{\mathcal{J}^+(p_{0})}\int_{\mathcal{J}^-(q)\cap \mathcal{J}^+(p_{0})} \big((\D^pu,\tilde V_{p}(q)), \D^q\phi\big) \mu(p)\wedge \mu(q)\\
&+&\int_{\mathcal{J}^+(p_{0})}\int_{\mathcal{C}^-(q)\cap\mathcal{
C}^+(p_{0})} \big(\nabla^p \Gamma_{0}\cdot u,\tilde U_{p}(q)), \D^q\phi\big) \mu_{\Gamma_{0},\Gamma_{q}}(p)\wedge \mu(q)\\
&+&\int_{\mathcal{J}^+(p_{0})}\int_{\mathcal{J}^-(q)\cap\mathcal{C}^+(p_{0})} \big((\nabla^p \Gamma_{0}\cdot u,\tilde V_{p}(q)), \D^q\phi\big)  \mu_{\Gamma_{0}}(p)\wedge \mu(q).
\end{eqnarray*}

Finally, the duality bracket \eqref{justif2} is:
\begin{gather*}
(uH^+(\Gamma_{0}),\phi)=
\left(\left(\int_{\mathcal{S}(q)} \big(\D^pu,\tilde U_{p}(q)\big) \mu_{\Gamma_{q}}(p)+\int_{\mathcal{V}(q)} \big(\D^pu,\tilde V_{p}(q)\big) \mu_{\Gamma_{q}}(p)\right.\right.\\
\left.\left.+\int_{\sigma(q)} \big(\nabla^p \Gamma_{0}\cdot u,\tilde U_{p}(q)\big) \mu_{\Gamma_{0},\Gamma_{q}}(p)+\int_{\mathcal{D}(q)} \big(\nabla^p \Gamma_{0}\cdot u,\tilde V_{p}(q)\big)  \mu_{\Gamma_{0}}(p)\right)H^+(\Gamma_{0}), \D^q\phi\right)_{q}
\end{gather*}
which means that, in the sense of distributions, $u$ satisfies, using the symmetry of the operator $\D^q$:
\begin{gather*}
uH^+(\Gamma_{0})=
\D^q\left(\left(\int_{\mathcal{S}(q)} \big(\D^pu,\tilde U_{p}(q)\big) \mu_{\Gamma_{q}}(p)+\int_{\mathcal{V}(q)} \big(\D^pu,\tilde V_{p}(q)\big) \mu_{\Gamma_{q}}(p)\right.\right.\\
\left.\left.+\int_{\sigma(q)} \big(\nabla^p \Gamma_{0}\cdot u,\tilde U_{p}(q)\big) \mu_{\Gamma_{0},\Gamma_{q}}(p)+\int_{\mathcal{D}(q)} \big(\nabla^p \Gamma_{0}\cdot u,\tilde V_{p}(q)\big)  \mu_{\Gamma_{0}}(p)\right)H^+(\Gamma_{0})\right).\text{\fin}
\end{gather*}

A direct application of the previous theorem is the first integral formula for the characterictic Cauchy problem:
\begin{proposition}\label{repintegral}
Let $u$ be a smooth solution of:
\begin{equation*}
\D u=0
\end{equation*}
Then $u$ can be expressed in $\mathcal{J}^+(p_{0})$ in function of its restriction to the cone $\mathcal{C}^+(p_{0})$ by:
\begin{gather*}
u|_{\mathcal{J}^+(p_0)}=
\D^q\left(\left(\int_{\sigma(q)} \big(\nabla^p \Gamma_{0}\cdot u,\tilde U_{p}(q)\big) \mu_{\Gamma_{0},\Gamma_{q}}(p)+\int_{\mathcal{D}(q)} \big(\nabla^p \Gamma_{0}\cdot u,\tilde V_{p}(q)\big)  \mu_{\Gamma_{0}}(p)\right)H^+(\Gamma_{0})\right).
\end{gather*}
\end{proposition}
\begin{remark}\begin{enumerate}
\item This formula is not the final stage of our calculation; the fact that it only depends on initial conditions will be stated later. This is the purpose of the next subsection.
\item The vector $\nabla \Gamma_0$ being null along the cone $\mathcal{C}^+(p_0)$, Clifford multiplying with $\hat\nabla \Gamma_0$  means in fact contracting with the spinor form of $\nabla \Gamma_0$; a direct consequence of this is the fact the Clifford product of the 4-components Dirac spinors with $\nabla \Gamma_0$ does only involve the two components, $u_{0}$ and $u^{1'}$. The two remaining components are recovered using the constraints equations (cf. lemma \ref{scalarrelation} below).
\end{enumerate}
\end{remark}

\subsection{Integral formula}

The integral formula is derived in three steps:
\begin{itemize}
\item construction of the appropriate geometric tools (derivation of measures, spin basis);
\item interversion of the integral and the Dirac operator;
\item and finally obtention of  an expression of the singular part in terms of geometric quantities and initial data.
\end{itemize}

\subsubsection{Geometric data on the cone}\label{geometricdataonthecone}

This section is devoted to the calculation of the relevent geometric quantities for the intersection of $\mathcal{C}^+(p_0)\cap\mathcal{C}^-(q)=\sigma(q)$ for a given point $q$ in the future of $p_0$. This is widely inspired by section 4.14 of \cite{Penrose:1986fk}. There are also some calculations of interest in the work of Frittelli, Newman and al (\cite{MR1674227}, for instance) and  Nurowski -- Robinson(\cite{nr00}).  This kind of calculation is also very common in the study of Ricci flows.

We first choose a parallely transported vector field $l$ along the null cone $\mathcal{C}^+(p_0)$:
\begin{equation*}
\nabla_l l=0. 
\end{equation*}

We then consider, for a given point $q$ in $\mathcal{J}^+(p_{0})$, a point $p$ in $\sigma(q)$.  We construct at $p$ a Newman-Penrose tetrad:
\begin{enumerate}
\item the first null vector is the vector $l(p)$ at $p$;
\item $n(p)$ is chosen on the future oriented null geodesic from $p$ to $q$ such that $g(l,n)=1$;
\item we complete the basis by taking a pair of complex null vectors $m(p)$ and $\overline{m}(p)$ in the orthogonal of the vector space generated by $(l,n)$ such that $g(m,\overline{m})=-1$.
\end{enumerate} 
A Newman-Penrose tetrad is then obtained at each point $q'$ on the cone $\mathcal{C}^-(q)$: let $p'$ be the point in $\sigma(q)$ lying on the unique null geodesic from $q'$ to $q$; the Newman-Penrose tetrad is obtained in $q'$ by parallely transporting the one at $p'$ along the unique null geodesic from $p'$ to $q'$.

\begin{remark}\label{geometricobstructions}
This construction cannot be realized globally on the intersection $\mathcal{C}^+(p_{0})\cap \mathcal{C}^-(q)=\sigma(q)$ which has the topology of $\mathbb{S}^2$. It will be necessary to make this construction on two different open sets and then glue these constructions together to obtain the result which only depends on $l$ and $n$. We assume then that the construction is done on one open set.
\end{remark}

This choice of Newman Penrose tetrad gives us:
\begin{itemize}
\item a basis of $T \Omega\otimes \C$ and, consequently, up to a sign, a spin basis of $\ds^A$ that will be denoted by $(o^A, \iota^A)$;
\item if $q$ is fixed first and $p$ is chosen on $\sigma(q)$, the vectors $m$ and $\overline {m}$ span the tangent plane to $\sigma(q)$ at $p$:
$$
T_p\sigma(q)=\{\lambda \overline{m}+\overline{\lambda} m | \lambda \in \C \};
$$
due to obvious topological obstructions (see remark \ref{geometricobstructions}), this construction cannot be extended globally to all $\sigma(q)$.
\item the choice of $l$, which is parallely transported along the generators of $\mathcal{C}^+(p_0)$, and $n$, which is parallely transported along the generators of $\mathcal{C}^-(q)$, gives rise to two affine parameters $r_0$ and $r$ along the null geodesics on these two cones.
\item these two affine parameters give rise to two parametrizations by the sphere $S^2$ of $\sigma(q)$ using the exponential map at $p_0$ and $p$ respectively:
\begin{equation*}
\begin{array}{lcccl}
\exp_{p_{0}}: &S^2 &\longrightarrow& \Omega\\ 
&\omega& \longmapsto & \exp_{p_{0}}(r_{0}(\omega)\omega)
\end{array}
\end{equation*}
and
\begin{equation*}
\begin{array}{lcccl}
\exp_{p}: &S^2 &\longrightarrow& \Omega\\ 
&\omega& \longmapsto & \exp_{p}(r(\omega)\omega)
\end{array}
\end{equation*}
\end{itemize}

Let $q$ be a point fixed in $\mathcal{J}^+(p_0)$. We consider a point $p$ on $\sigma(q)$. In a neighborhood of $p$, on $\sigma(q)$, is defined a Newman-Penrose tetrad $(l,n,m,\overline{m})$. The dual basis in $T\Omega^\ast\otimes \C$ is denoted by $(L,N,M,\overline{M})$ for which the following lemmata are true:
\begin{lemma}\label{meancurvature}
The induced metric on $\sigma(p)$ is $-2M \overline{M}$, the volume form $\frac{1}{2i} M\wedge \overline{M}$ and the mean curvature vector:
\begin{equation*}
H=2(\rho' l+\rho n)
\end{equation*}
where $\rho$ and $\rho'$ are the real spin coefficients:
$$
\rho= -(l,\nabla_{\overline{m}}m) \text{ and } \rho'= -(n,\nabla_{m}\overline{m})
$$
\end{lemma}
\proof: These results are straightforward consequences of the presentation concerning two-surfaces in \cite{Penrose:1986fk} (section 4.14, proposition 4.14.2 sqq.)\\
The reality of the spin coefficients is stated in proposition (4.14.2) of \cite{Penrose:1986fk}, whenever $l$ and $n$ are orthogonal to a spacelike 2-surface (here $\sigma(q)$).\\
Since $(m, \overline{m})$ span $T\sigma(q)$, the second fundamental form is:
\begin{equation*}
\forall (X,Y)\in T\sigma(q), II(X,Y)=(\nabla_X Y, n)l+(\nabla_X Y,l)n
\end{equation*}
so that the mean curvature vector is:
\begin{eqnarray*}
H&=&-(II(m, \overline{m})+II(\overline{m},m))\\
&=&-\left(((\delta'm,n)+(\delta \overline m,n))l+((\delta'm,l)+(\delta \overline m,l))n\right)
\end{eqnarray*}
Since (see \cite{Penrose:1986fk} (4.5.28) together with (4.5.29)):
\begin{eqnarray*}
\delta' m&=&(\beta-\overline{\alpha})m-\overline{\rho}'l-\rho n\\
\delta \overline{m}&=&(\overline{\alpha}-\beta)m-\rho'l-\overline{\rho} n
\end{eqnarray*}
and since $\rho$ and $\rho'$ are real, we obtain:
\begin{equation*}
H=2(\rho'l+\rho n).\text{\fin}
\end{equation*}

In order to compute the Leray forms associated with the distance function, we use the expressions of the gradients of the distance functions $\Gamma_0$ and $\Gamma_q$:
\begin{equation}
\nabla^p\Gamma_0(p)=2r_0 l(p) \text{ and } \nabla^p \Gamma_q(p)=2r n(p)
\end{equation}
The Leray forms can then be expressed using the dual basis of the chosen Newman-Penrose basis.
\begin{proposition}\label{expmeasures}
The Leray forms $\mu_{\Gamma_0}$, $\mu_{\Gamma_q}$ and $\mu_{\Gamma_0,\Gamma}$ are:
\begin{equation*}
\begin{array}{lcl}
\mu_{\Gamma_0}&=&\displaystyle{\frac{1}{2ir_0}N\wedge M\wedge \overline{M}}\\
\mu_{\Gamma_q}&=&\displaystyle{\frac{1}{2ir}L\wedge M\wedge \overline{M}}\\
\mu_{\Gamma_0,\Gamma}&=&\displaystyle{\frac{1}{4ir_0r}M\wedge \overline{M}}=\displaystyle{\frac{1}{4r_0r}\mu_{\sigma(p)}}
\end{array}
\end{equation*}
where $\nabla^p \Gamma_0(p)=2r_0 l(p)$ and $\nabla^p \Gamma_q(p)=2r n(p)$
\end{proposition}
\proof:
The volume form on $\Omega$ can be expressed in terms of the Newman-Penrose tetrad as:
\begin{equation*}
\mu=\frac{1}{i} L\wedge N\wedge M\wedge \overline{M}
\end{equation*}
so that, since $\ud \Gamma_{0}=2r_0 L$ and $\ud \Gamma_q = 2 r N$, we obtain immediately:
\begin{equation*}
\mu_{\Gamma_0}=\displaystyle{\frac{1}{2ir_0}N\wedge M\wedge \overline{M}}\text{ and }
\mu_{\Gamma_q}=\displaystyle{\frac{1}{2ir}L\wedge M\wedge \overline{M}}.
\end{equation*} 
The calculation of $\mu_{\Gamma_0, \Gamma}$ is obtained through the factorization given by Fubini's theorem:
\begin{equation*}
\forall \phi \in C^{\infty}_0(\Omega\times \Omega), \int_{A}\phi \mu_{\Gamma_0}(p)\wedge \mu_{\Gamma}(q)=\int_{\mathcal{J}^{+}(p_0)}\int_{\sigma(q)}\phi\mu_{\Gamma_0,\Gamma_q}(p)\wedge \mu(q),
\end{equation*}
where $A$ is the set $\{(p,q)|p\in \mathcal{C}^+(p_{0}) \text{ and } q \in \mathcal{C}^+(p) \}$. We get the desired expression of $\mu_{\Gamma_0,\Gamma_q}$:
\begin{equation*}
\mu_{\Gamma_0,\Gamma_q}=\frac{1}{4ir_0r}M\wedge \overline{M}=\frac{1}{4r_0 r}\mu_{\sigma(q)}.\text{\fin}
\end{equation*}

The next step consists in determining the variation of the metric $\mu_{\sigma(q)}$ when $q$ is in $\mathcal{I}^+(p_0)$. We first establish the technical lemma:
\begin{lemma}\label{devmes} Let $(\mathcal{N},h)$ be a smooth semi-riemanian manifold with metric $h$ and Levi-Cevita connexion $D$; let  $X$ be a smooth vector field on $\mathcal{N}$.
Let  $(\mathcal{M}_{p}, g)$ be a submanifold of $\mathcal{N}$ such that the $g$ metric induced by $h$ is non degenerate and depending smoothly on a parameter $p$ in $\mathcal{N}$ in the sense that there exists a smooth manifold $\Sigma$ and a smooth map $f:\mathcal{N} \times \Sigma \longrightarrow \mathcal{N}$ which satisfies: $f(p,\star)$ is an immersion and $f(p,\Sigma)=\mathcal{M}_p$.\\
We denote by $\mu_p$ the induced volume form on $\mathcal{M}_{p}$.\\
Then:
\begin{equation*}
D^p_X \mu_p=-h(H,X)\mu_p
\end{equation*}
where $H$ is the mean curvature vector field on $\mathcal{M}_p$.
\end{lemma}
\proof: The Levi-Cevita connection induced by $g$ on $\mathcal{M}_{p}$ is denoted $\nabla$. \\
Let $p$ be a point in $\mathcal{N}$ and $q$ a point in $\mathcal{M}_{p}$. We choose around $q$ a map \\$(V,(x^1,x^2,\dots,x^n),(x^{n+1},\dots,x^{n+k}))$ normal at $q$ and $\mathcal{M}_{p}\cap V=\{x^{n+1}=\dots=x^{n+k}=0\}$ such that, at $q$:
\begin{equation}
\nabla_{\partial_{x^i}}\partial_{x^i}=0.
\end{equation}
The volume form on $\mathcal{M}_{p}$ around $q$ can be expressed:
\begin{equation*}
\mu_p=|\text{det}(g_{ij})|^{\frac{1}{2}}\ud x^1\wedge \ud x^2 \wedge \dots \wedge \ud x^n.
\end{equation*}
We calculate the derivative:
\begin{equation*}
D_X \mu_p=\text{Sign}(\det(g_{ij}))\frac{g^{ij}D_X g_{ij}}{2\det(g_{ij})}|\text{det}(g_{ij})|^{\frac{1}{2}}\ud x^1\wedge \ud x^2 \wedge \dots \wedge \ud x^n
\end{equation*}
and then evaluate at $q$, where the coordinate system is normal:
\begin{equation*}
D_X\mu_p=\sum_{i=1}^n\frac{1}{2}\varepsilon_i D_X g(\partial_{x^i},\partial_{x^i}) \mu_p
\end{equation*}
with $\varepsilon_i=g(\partial_{x^i},\partial_{x^i})$. Since the connection $D$ on $\mathcal{N}$ is metric, we get:
\begin{eqnarray*}
D_X \mu_p&=& \sum_{i=1}^n\varepsilon_ih(D_{X}\partial_{x^i},\partial_{x^i})\mu_p\\
&=& \sum_{i=1}^n\varepsilon_i h(D_{\partial_{x^i}}X+[X,\partial_{x^i}],\partial_{x^i})\mu_p\\
&=& \sum_{i=1}^n \varepsilon_i ( D_{\partial_{x^i}}h(X,\partial_{x^i})-h(X,D_{\partial_{x^i}}\partial_{x^i})+h([X,\partial_{x^i}],\partial_{x^i}))\mu_p
\end{eqnarray*}
Since $D_{\partial_{x^i}}\partial_{x^i}= \nabla_{\partial_{x^i}}\partial_{x^i}+II(\partial_{x^i},\partial_{x^i})$, we finally obtain:
\begin{eqnarray*}
\nabla_X \mu_p&=&-h(X,H)\mu_p + \sum_{i=1}^n \varepsilon_i ( D_{\partial_{x^i}}h(X,\partial_{x^i})-h(X, \nabla_{\partial_{x^i}}\partial_{x^i})+h([X,\partial_{x^i}],\partial_{x^i})\mu_p.
\end{eqnarray*}
We then notice that, for $i$ in $\{t1,n\}$:
$$
h([X,\partial_{x^i}],\partial_{x^i})=-\eps_{i}\frac{\partial X^{i}}{\partial x^{i}} \text{ and }  D_{\partial_{x^i}}h(X,\partial_{x^i})=\eps_{i}\frac{\partial X^{i}}{\partial x^{i}}.
$$
Since $\nabla_{\partial_{x^{i}}}\partial_{x^{i}}$ is vanishing at $q$, the only remaining term is:
$$
\nabla_X \mu_p=-h(X,H)\mu_p.\text{\fin}
$$


A straightforward application of this lemma is the proposition:
\begin{proposition}\label{diracderivation1}
Let $f:\Omega^2\rightarrow \ds_{Dirac} $ be a smooth mapping.\\
Then the following formula holds:
\begin{equation*}
\D^q \int_{\sigma(q)}f(q,p)\mu_{\sigma(q)}(p)=\int_{\sigma(q)}\D^qf(q,p)+\nabla^{p} r_0\cdot \hat\nabla_{l}^{p}f(q,p)- 2\rho\hat\nabla^{q} r_{0}\cdot f(q,p)\mu_{\sigma(q)}(p)
\end{equation*}
where $r_0$, being a function of both $p$ and $q$, satisfies $\nabla \Gamma_0=2r_0 l$.
\end{proposition}
\proof: Let $V$ a vector field on $\Omega$. We work with the exponential map  centered at $p_0$. $\sigma(q)$ can then be parametrized by $S^2$:
\begin{equation*}
\omega\mapsto \exp_{p_0}(r_0(q, \omega) \omega).
\end{equation*}
Let us consider the variation of $\sigma(q)$ defined by, for some positive $\eps$:
\begin{equation*}
\begin{array}{lcl}
]-\eps,\eps[\times S^2 &\rightarrow &\Omega\\
(t,\omega)&\mapsto& \exp_{p_0}(r_0(q+tV,\omega)\omega)
\end{array}.
\end{equation*}
Since 
\begin{eqnarray*}
\nabla_{V}^q\big ( f(q,p) \big)&=&\nabla_{V}^q\big( f(q, \exp_{p_{0}}(r_{0}(q, \omega)\omega)\big)\\
&=&\nabla_{V}^qf(q,p)+\left(\frac{\ud }{\ud t}f(q, \exp_{p_{0}}(r_{0}(q+tV, \omega)\omega)\right)\big|_{t=0}
\end{eqnarray*}
and 
$$
\left(\frac{\ud }{\ud t}f(q, \exp_{p_{0}}(r_{0}(q+tV, \omega)\omega)\right)\big|_{t=0}=\nabla^{q}_V r_0 \nabla_{l}^{p}f(q,p),
$$
this gives, using lemma \ref{meancurvature}:
\begin{equation*}
\nabla_V^q \int_{\sigma(q)}f(q,p)\mu_{\sigma(q)}(p)=\int_{\sigma(q)}\nabla_V^qf(q,p)+\nabla^{q}_V r_0 \nabla_{l}^{p}f(q,p)- 2\rho\nabla^{q}_V r_{0} f(q,p)\mu_{\sigma(q)}
\end{equation*}
so that, when choosing an orthonormal basis $(e_i)$ on $\Omega$, we obtain:
\begin{equation*}
\D^q \int_{\sigma(q)}f(q,p)\mu_{\sigma(q)}(p)=\int_{\sigma(q)}\D^qf(q,p)+\hat\nabla^{q} r_0\cdot \nabla_{l}^{p}f(q,p)- 2\rho\hat\nabla^{q} r_{0}\cdot f(q,p)\mu_{\sigma(q)}(p).\text{\fin}
\end{equation*}

We finally establish the following proposition:

\begin{proposition}\label{diracderivation2}
Let $f:\Omega^2\rightarrow \ds_{Dirac} $ be a smooth mapping.\\
Then the following formula holds:
\begin{equation*}
\D^q \int_{\mathcal{D}(q)}f(q,p)\mu_{\Gamma_{0}}(p)=\int_{\mathcal{D}(q)}\D^qf(q,p)\mu_{\Gamma_{0}}(p)+\int_{\sigma(q)}\hat\nabla^{q} r_0\cdot f(q,p)\frac{\mu_{\sigma(q)}(p)}{2r},
\end{equation*}
where $r_0$ and $r$, being functions of both $p$ and $q$, satisfy $\nabla^p \Gamma_0=2r_0 l$ and $\nabla^p \Gamma_{q}=2r n$.
\end{proposition}
\proof: We use exactly the same method as in the proof of proposition \ref{diracderivation1}. Using the parametrization of the exponential map centered at $p_{0}$, we have:
\begin{eqnarray*}
\D^q \int_{\mathcal{D}(q)}f(q,p)\mu_{\Gamma_{0}}(p)&=&\D^q\int_{S^2}\int_{0}^{r_{0}(q,\omega)}f(q,r\omega)\frac{\ud r}{2r}k(\omega,r)\ud \omega_{S^2}\\
&=&\int_{\mathcal{D}(q)}\D^qf(q,p)\mu_{\Gamma_{0}}(p)+\int_{S^2}\hat \nabla^q r_{0} \cdot f(q,r\omega) \frac{k(\omega,r)\ud \omega_{S^2}}{2r}\\
&=&\int_{\mathcal{D}(q)}\D^qf(q,r\omega)\mu_{\Gamma_{0}}(p)+\int_{\sigma(q)}\hat\nabla^{q} r_0\cdot f(q,p)\frac{\mu_{\sigma(q)}(p)}{2r}.\text{\fin}
\end{eqnarray*}

\subsubsection{Derivation of the integral formula}

We now consider the characteristic Cauchy problem on $\Omega$:
\begin{equation}\label{ccps}
\left\{\begin{array}{lcl}
\D u &=&0 \text{ on } \mathcal{J}^+(p_0)\\
u&=&\theta \text{ on  } \mathcal{C}^+(p_0)
\end{array}\right.
\end{equation}
where $\theta$ is a smooth spinor field on $\mathcal{C}^+(p_0)$, whose support does not encounter the vertex of the cone and satisfies the constraint equations given by lemma \ref{scalarrelation}.
\begin{remark} The term "spinor field on the cone" must be understood as "trace on the cone" of a Dirac spinor field on $\Omega$ and not as a spinor field  constructed as spinors on the manifold $\mathcal{C}^+(p_0)$.
\end{remark}
The basis constructed in the previous section is used to split the spinors:
\begin{equation}
\begin{array}{lcl}
\theta=\xi^{\textbf{I}'}\eps_{\textbf{I}'}^{A'}+\zeta_{\textbf{I}}\eps^{\textbf{I}}_{A}=\xi^{0'}\overline{o}^{A'}+\xi^{1'}\overline{\iota}^{A'}+\zeta_0(-\iota_{A})+\zeta_1o_{A}.
\end{array}
\end{equation}
$u$ will be split on $\ds_A\otimes \ds^{A'}$:
\begin{eqnarray*}
u&=&\phi_A+\psi^{A'}\\
u&=&\phi_0(-\iota_{A})+\phi_1o_{A}+\psi^{0'}\overline{o}^{A'}+\psi^{1'}\overline{\iota}^{A'}.
\end{eqnarray*}

The solution of \eqref{ccps} can be written in function of its data on the cone and the basis $(o^A,\iota^A)$:
\begin{theorem}\label{repweyl}
Let $u$ be a solution of \eqref{ccps}. Then, for any $q$ in $\mathcal{J}^+(p_0)$:
\begin{gather*}
u(q)=\int_{\sigma(q)}\left(\frac{k_p(q)}{r}\right)(\nabla_l^p\xi^{1'}(p)-\rho\xi^{1'}(p))\left(\nabla^{AA'}_qr_0\right) o_{A}(q)\mu_{\sigma(q)}(p)\\
+\int_{\sigma(q)}\xi^{1'}(p)\nabla^{AA'}\left(\overline k_p(q)o_{A}(q)\right)\mu_{\sigma(q)}(p)\\
+\int_{\sigma(q)}\left(\frac{k_p(q)}{r}\right)(\nabla_l^p\zeta_0(p)-\rho\zeta_0(p))\left(\nabla_{AA'}^qr_0\right) \overline o^{A'}(q)\mu_{\sigma(q)}(p)\\
+\int_{\sigma(q)}\zeta_0(p)\nabla_{AA'}\left(\overline k_p(q)\overline o^{A'}(q)\right)\mu_{\sigma(q)}(p)+\int_{\sigma(q)}\hat \nabla^qr_{0} \cdot \big(\nabla^p \Gamma_{0}\cdot u,\tilde V_{p}(q)\big) \frac{\mu_{\sigma(q)}(p)}{2r}\\
+\int_{\mathcal{D}(q)}(\D^p \tilde V_q,\hat \nabla^p \Gamma_0 \cdot u)\mu_{\Gamma_0}(p)
\end{gather*}
\end{theorem}
\begin{remark} It is possible to obtain a representation formula for the Goursat problem for the Weyl equation:
$$
\nabla^{AA'}\phi_{A}=0
$$
by projecting the solution obtained in theorem \ref{repweyl} on the subspace of Dirac spinors $\ds_{A}$.
\end{remark}
\proof:  Let $q$ be a point in $\mathcal{J}^+(p_0)$. Using proposition \ref{repintegral}, proposition \ref{expmeasures} and proposition \ref{diracderivation2}, we have:
\begin{gather*}
u(q)=
\D^q\left(\int_{\sigma(q)}\left(\hat \nabla^p \Gamma_0\cdot u,\tilde U_p\right)_p\displaystyle{\frac{1}{4r_0r}\mu_{\sigma(q)}}\right)+\int_{\sigma(q)}\hat \nabla^qr_{0}\cdot  \big(\nabla^p \Gamma_{0}\cdot u,\tilde V_{p}(q)\big)   \frac{\mu_{\sigma(q)}(p)}{2r}\\+\int_{\mathcal{D}(q)}(\hat \nabla^p \Gamma_0 \cdot u,\D^q \tilde V_q)\mu_{\Gamma_0}(p).
\end{gather*}
The bracket in the first integral can be calculated as follows:
\begin{eqnarray*}
\left(\tilde U_p,\displaystyle{\frac{\hat \nabla^p \Gamma_0\cdot u}{2rr_0}}\right)_p&=&k_p(q)\left(\tau_p(q),\displaystyle{\frac{1}{2r}}n\cdot u\right)_p\\
&=&i\sqrt{2}k_p(q)\left(\tau_p(q),\displaystyle{\frac{1}{2r}}\left(-l^{AA'} \phi_A+l_{AA'} \psi^{A'}\right)\right)_p\\
&=&i\sqrt{2}k_p(q)\left(\tau_p(q),\displaystyle{\frac{1}{2r}}\left(-o^{A}\overline{o}^{A'} \phi_A+o_{A}\overline{o}_{A'} \psi^{A'}\right)\right)_p\\
&=&i\sqrt{2}k_p(q)\left(\tau_p(q),\displaystyle{\frac{1}{2r}}\left(-\phi_0\overline{o}^{A'}+\psi^{1'}o_{A}\right)\right)_p\\
&=&i\sqrt{2}k_p(q)\displaystyle{\frac{1}{2r}}\left(-\phi_0(p)\overline{o}^{A'}(q)+\psi^{1'}(p)o_{A}(q)\right)\\
\end{eqnarray*}

We then use proposition \ref{diracderivation1} to calculate the first integral:
\begin{gather*}
\D^q\left(\int_{\sigma(q)}\left(\tilde U_p,\hat \nabla^p \Gamma_0\cdot u\right)_p\displaystyle{\frac{1}{4r_0r}}\mu_{\sigma(q)}\right)\\=i\sqrt{2}\D^q\left(\int_{\sigma(q)}k_p(q)\left(-\phi_0(p)\overline{o}^{A'}(q)+\psi^{1'}(p)o_{A}(q)\right)\displaystyle{\frac{1}{2r}\mu_{\sigma(q)}}\right)\nonumber\\
=i\sqrt{2}\int_{\sigma(q)}\D^qk_p(q) \left(\displaystyle{\frac{-\phi_0(p)\overline{o}^{A'}(q)}{2r}}+\displaystyle{\frac{\psi^{1'}(p)o_{A}(q)}{2r}}\right)\\
+\hat\nabla^qr_0\cdot \nabla^p_l k_p(q)\left(\displaystyle{\frac{-\phi_0(p)\overline{o}^{A'}(q)}{2r}}+\displaystyle{\frac{\psi^{1'}(p)o_{A}(q)}{2r}}\right)\\
-\displaystyle{\frac{1}{r}}\rho\hat\nabla^q r_0\cdot\left(-\phi_0(p)\overline{o}^{A'}(q)+\psi^{1'}(p)o_{A}(q)\right)\mu_{\sigma(q)}(p).
\end{gather*}
In order to simplify the calculation, the previous formula is projected on $\ds_A$. The singular part on this subspace is written, after expansion:
\begin{eqnarray*}
A&=&i\sqrt{2}\int_{\sigma(q)}i\sqrt{2}\nabla^q_{AA'}\left(k_p(q)\left(\displaystyle{\frac{-\phi_0(p)\overline{o}^{A'}(q)}{2r}}\right)\right)\\&&+i\sqrt{2}\nabla^q_{AA'}r_0\nabla^p_l\left(k_p(q)\left(\displaystyle{\frac{-\phi_0(p)\overline{o}^{A'}(q)}{2r}}\right)\right)\\
&&-(i\sqrt{2})k_p(q)\rho\nabla^q_{AA'}r_0\left(\displaystyle{\frac{-\phi_0(p)\overline{o}^{A'}(q)}{r}}\right)\mu_{\sigma(q)}(p).\nonumber\\
&=&-\int_{\sigma(q)}\nabla^q_{AA'}\left(k_p(q)\left(\displaystyle{\frac{\phi_0(p)\overline{o}^{A'}(q)}{r}}\right)\right)+\nabla^q_{AA'}r_0\nabla^p_l\left(k_p(q)\left(\displaystyle{\frac{\phi_0(p)\overline{o}^{A'}(q)}{r}}\right)\right)\nonumber \\&&-2k_p(q)\rho\nabla^q_{AA'}r_0\left(\displaystyle{\frac{-\phi_0(p)\overline{o}^{A'}(q)}{r}}\right)\mu_{\sigma(q)}(p).\nonumber\\
&=&-\int_{\sigma(q)}\phi_0(p)\nabla^q_{AA'}\left(k_p(q)\left(\displaystyle{\frac{\overline{o}^{A'}(q)}{r}}\right)\right)+\nabla^q_{AA'}r_0\nabla^p_l\left(k_p(q)\left(\displaystyle{\frac{\phi_0(p)\overline{o}^{A'}(q)}{r}}\right)\right)\nonumber \\&&-2k_p(q)\rho\nabla^q_{AA'}r_0\left(\displaystyle{\frac{-\phi_0(p)\overline{o}^{A'}(q)}{r}}\right)\mu_{\sigma(q)}(p).\nonumber\\
\end{eqnarray*}
Expanding all the products:
\begin{eqnarray*}
A&=&-\int_{\sigma(q)}\Big(\phi_{0}(p)\left(\overline o^{A'}(q)\nabla^q_{AA'}\left(\frac{k_p(q)}{r}\right)+\left(\frac{k_p(q)}{r}\right)\nabla_{AA'}^q\overline o^{A'}(q) \right)\\
&&+\left(\left(\frac{k_p(q)}{r}\right) \nabla_l^p \phi_0+\nabla_l^p\left(\frac{k_p(q)}{r}\right)  \phi_0\right) \left(\nabla_{AA'}^qr_0\right) \overline o^{A'}(q)\\
&&\left.-2\left(\frac{k_{q}(p)}{r}\right)\rho(\nabla^q_{AA'}r_0)\phi_0(p)\overline{o}^{A'}(q)\right)\mu_{\sigma(q)}(p)\\
&=&-\int_{\sigma(q)}\left(\frac{k_p(q)}{r}\right)(\nabla_l^p\phi_0(p)-2\rho\phi_0(p))\left(\nabla_{AA'}^qr_0\right) \overline o^{A'}(q)\mu_{\sigma(q)}(p)\\
&&-\int_{\sigma(q)}\phi_0(p)\nabla_{AA'}\left(\left(\frac{k_p(q)}{r}\right)\overline o^{A'}(q)\right)\mu_{\sigma(q)}(p).
\end{eqnarray*}
Since the quantities which appear in the integral are the restriction of $u$ and its tangential derivative along the null cone $\mathcal{C}^+(p_0)$, $\phi$ can be replaced in the integral by the data of the Goursat problem $\zeta_0$:
\begin{gather*}
\int_{\sigma(q)}\left(\frac{k_p(q)}{r}\right)(\nabla_l^p\zeta_0(p)-2\rho\zeta_0(p))\left(\nabla_{AA'}^qr_0\right) \overline o^{A'}(q)\mu_{\sigma(q)}(p)\\
+\int_{\sigma(q)}\zeta_0(p)\nabla_{AA'}\left(\left(\frac{k_p(q)}{r}\right)\overline o^{A'}(q)\right)\mu_{\sigma(q)}(p).
\end{gather*}

We obtain the complete formula for Dirac spinors by adding the corresponding quantity on $\ds^{A'}$, meaning:
\begin{gather*}
\int_{\sigma(q)}\left(\frac{k_p(q)}{r}\right)(\nabla_l^p\xi^{1'}0(p)-2\rho\xi^{1'}(p))\left(\nabla^{AA'}_qr_0\right) o_{A}(q)\mu_{\sigma(q)}(p)\\
+\int_{\sigma(q)}\phi_0(p)\nabla^{AA'}\left(\left(\frac{k_p(q)}{r}\right)o_{A}(q)\right)\mu_{\sigma(q)}(p).
\end{gather*}
Noticing that the calculation has been done for $(\tilde U_{p}(q),\hat \nabla \Gamma_{0}\cdot u)$ in order to use the definition of $\tau(p,q)$, we obtain the complete formula using the antisymmetry of the symplectic product.\fin

It is now possible to obtain the formula established by Penrose in \cite{p80} in the Minkowski case:
\begin{theorem}[Penrose]
Let $u$ be a solution of \eqref{ccps}.\\
Then, for all $q$ in $\mathcal{J}^+(p_0)$, $u$ can be written:
\begin{gather*}
u(q)=\int_{\sigma(q)}\frac{1}{2\pi r}(\nabla_l^p\xi^{1'}0(p)-2\rho\xi^{1'}(p))\left(\nabla^{AA'}_qr_0\right) o_{A}(q)\mu_{\sigma(q)}(p)\\
+\int_{\sigma(q)}\frac{1}{2\pi r}(\nabla_l^p\zeta_0(p)-2\rho\zeta_0(p))\left(\nabla_{AA'}^qr_0\right) \overline o^{A'}(q)\mu_{\sigma(q)}(p)
\end{gather*}
\end{theorem}
\begin{remark}
First of all, the meaning in the context of a flat space of the choice of the basis constructed in the previous section should be made precise:
\begin{itemize}
\item the spinor $o^A$ is chosen to be constant on the null generators of the cone; the affine parameter $r_0$ is measured with respect to the vector $l^a=o^A\overline o ^{A'}$;
\item a direction on the cone $\mathcal{C}^+(p_0)$ being given together with a point $q$ in $\mathcal{J}^+(p_0)$, let $p$ be the intersection of $\mathcal{C}^-(q)$ with this direction on the null cone from $p_0$; the spinor $\iota ^A$ is chosen so that $n^a=\iota ^A\overline \iota^{A'} $ is colinear to the vector $\vec{pq}$ and satisfies: $o_A\iota^A=1$; the affine parameter $r$ is measured with respect to the vector $n^a=\iota ^A\overline \iota^{A'}$;
\item the basis is completed by the two vectors $m^a=o^A\overline\iota^{A'}$ and $\overline m^a=\iota^A\overline o^{A'}$.
\end{itemize}
This construction is the "flat" version of the one made using parallel transport.
\end{remark}
\proof: As done in \cite{p80}, it is sufficient to remark, for a direction $\omega$ on the cone $\mathcal{C}^+(p_0)$:
\begin{equation*}
q= p_0+r_0 l^a(\omega)+r n^a(q,\omega),
\end{equation*}
which implies:
\begin{equation*}
\nabla^q r= l^a, \nabla^q r_0= n^a
\end{equation*}
and $k_p(q)=\frac{1}{2\pi}$.\fin
\begin{remark}
It is interesting to note that the term that carries the curvature information in the singular part is:
\begin{equation}\label{imp}
\nabla^q \left(\left(\frac{k_p(q)}{r}\right) o^A\right).
\end{equation}
It is somehow difficult to give a precise geometric interpretation to equation \eqref{imp}. Nevertheless, clues can be found in theorem $4.2.2$ in \cite{Friedlander:1975vn} that states that $\left(k/r\right)^2$ measures the growth rate of the measure $\mu_{\sigma(q)}$.
\end{remark}

\section{Generalization to higher spin}

In this section, we obtain an integral formula for solutions of the Goursat problem for the Dirac equation with arbitrary spin. The derivation of the formula is based on the representation formula for the Weyl equation which can be extracted from theorem \ref{repweyl}.

Let us consider the characteristic Cauchy problem for spin $\frac{n}{2}=s\geq1$ ($n$ being the number of indices of a spinor):
\begin{equation}\label{ccphs}
\left\{
\begin{array}{lcl}
\nabla^{AA'}u_{AB\dots F}&=&0 \text{ on } \mathcal{J}^+(p_0)\\
u_{00\dots 0}&=&\theta_{00\dots 0} \text{ on } \mathcal{C}^+(p_0)
\end{array}
\right.,
\end{equation}
where $u_{AB\dots F}$ satisfies the symmetry conditions:
\begin{equation*}
u_{AB\dots F}=u_{(AB\dots F)}.
\end{equation*}

First of all, it must be noted that, on an arbitrary curved space, the problem \eqref{ccphs} cannot be set if a consistency condition on the conformal curvature is not satisfied (\cite{MR0101150}, \cite{MR1683998} and \cite{n97} for the Rarita-Schwinger case for a treatment of the Cauchy problem). It is known that for the Dirac massless equation for low spin ($n\leq 1$, i.e. scalar wave, Dirac-Weyl and Maxwell equations) this condition is always satisfied. For higher spin, it is satisfied whenever the space-time is conformally flat. Nonetheless, it is expected that the method could be adapted to the Rarita-Schwinger case which requires the space-time to be Ricci flat.

\subsection{Generalization of Dirac equation to higher spin.}\label{generalization}

The construction that was made before for Dirac spinors is adapted here to spinors of higher valence so that the symmetry conditions of the Clifford multiplication and Dirac operator still hold.

Let us consider $\mathbb{E}$ the fibre bundle defined by:
\begin{equation*}
\mathbb{E}=\ds_{AB\dots F}\oplus\ds^{A'}_{\phantom{A'}G\dots I}.
\end{equation*}
This fibre bundle is equipped with the symplectic product obtained from $\eps$:
\begin{equation*}
\eps^{A\overline{A}}\eps^{B\overline{B}}\dots \eps^{F\overline{F}}\oplus\eps_{A'\overline{A}'}\eps^{G\overline{G}}\dots \eps^{I\overline{I}}
\end{equation*}
and a Clifford multiplication by vectors: if $u=\phi_{AB\dots F}+\psi^{A'}_{\phantom{A'}G\dots I}$ belongs to $\mathbb{E}$, we define $e_\textbf{a}\cdot u$, where $(e_\textbf{a})_{\textbf{a}=0,\dots, 3}$ is the basis constructed in subsection \ref{sectionspinor}:
\begin{equation*}
e_\textbf{a}\cdot u=-i\sqrt{2}g^{\textbf{a}AA'}\phi_{AB\dots F} +i\sqrt{2}g^{\textbf{a}}_{\phantom{\textbf{a}}AA'}\psi^{A'}_{\phantom{A'}G\dots I}.
\end{equation*}

We finally define on smooth sections $u=\phi_{AB\dots F}+\psi^{A'}_{\phantom{A'}B\dots F}$ of $\mathbb{E}$ the following operator (that will be denoted by $\D$ as the Dirac operator for Dirac spinors):
\begin{equation}
\D u= i\sqrt{2}\big(-\nabla^{AA'}\phi_{AB\dots F}+\nabla_{AA'}\psi^{A'}_{\phantom{A'}G\dots I}\big).
\end{equation}

The distributions on smooth sections of $\mathbb{E}$ are defined using the (non degenerate) symplectic product $\eps$ in the same way as in section \ref{analyticrequirements}. The duality bracket will still be denoted by $(,)_{\mathcal{D}'(\mathbb{E}),\mathcal{D}(\mathbb{E}))}$. Let $u=\phi_{AB\dots F}+\psi^{A'}_{\phantom{A'}G\dots I}$ and $v=\xi_{AB\dots F}+\zeta^{A'}_{\phantom{A'}G\dots I}$ be two smooth sections of $\mathbb{E}$. We have:
\begin{eqnarray*}
(v,u)_{\mathcal{D}'(\mathbb{E}),\mathcal{D}(\mathbb{E}))}&=&\int_{\Omega}\left(\eps^{A\overline{A}}\eps^{B\overline{B}}\dots \eps^{F\overline{F}}\xi_{AB\dots F}\phi_{\overline{AB\dots F}}+\eps_{A'\overline{A}'}\eps^{G\overline{G}}\dots \eps^{I\overline{I}}\zeta^{A'}_{\phantom{A'}G\dots I}\psi^{\overline{A}'}_{\phantom{\overline{A}'}\overline{G\dots I}}\right)\mu\\
&=&\int_{\Omega}\xi_{A\dots F}\phi^{A\dots F}+\zeta^{A'}_{\phantom{A'}G\dots I}\psi_{A'}^{\phantom{A'}G\dots I}\mu
\end{eqnarray*}

We finally extend the Schr\"odinger-Lichnerowicz formula to arbitrary spin:
\begin{proposition}[Schrödinger-Lichnerowicz formula for arbitrary spin]$\phantom{l}$\\
Let be $\psi_{F\dots I}$ a smooth section of $\ds_{F\dots I}$ (n indices).\\
Then the following formula holds:
\begin{eqnarray*}
\nabla_{BA'}\nabla^{FA'}\psi_{F\dots I}&=&\frac12\nabla_{HH'}\nabla^{HH'} \psi_{B
G\dots I}\\&-&X_{B}^{\phantom{B}F}\phantom{}_{F}^{\phantom{F}D}\psi_{DG\dots I}-X_{B}^{\phantom{B}F}\phantom{}_{G}^{\phantom{G}D}\psi_{FD\dots I}-\dots-X_{B}^{\phantom{B}F}\phantom{}_{I}^{\phantom{I}D}\psi_{FG\dots D}
\end{eqnarray*}
where $X_{ABCD}$ is the curvature spinor:
$$
X_{ABCD}=\frac14R_{AX'B}\phantom{}^{X'}\phantom{}_{CY'D}\phantom{}^{Y'},
$$
$R=R_{abcd}$ being the Riemann curvature tensor of $\Omega$.
\end{proposition}
\begin{remark} It must be noted that the potential of the operator $\D^2$, though linear, is no longer scalar and not even symmetric.
\end{remark}
\proof: the proof is almost the same as the proof of proposition \ref{slf12}:
\begin{eqnarray*}
\nabla_{BA'}\nabla^{FA'}\psi_{F\dots I}&=&\eps^{FC}\nabla_{BA'}\nabla^{A'}_C\psi_{F\dots I}\\
&=&\eps^{FC}\left(\nabla_{[B|A'}\nabla^{A'}_{|C]}\psi_{F\dots I}+\nabla_{(B|A'}\nabla^{A'}_{|C)}\psi_{F\dots I}\right)\\
&=&\frac{1}{2}\eps^{FC}\nabla_{HH'}\nabla^{HH'}\eps_{BC}\psi_{F\dots I}+ \eps^{FC}\nabla_{[B|A'}\nabla^{A'}_{|C]}\psi_{F\dots I}\\ 
&=&\frac{1}{2}\nabla_{HH'}\nabla^{HH'}\psi_{BG\dots I}+ \eps^{FC}\nabla_{[B|A'}\nabla^{A'}_{|C]}\psi_{F\dots I}
\end{eqnarray*}
The spinor $\psi_{F\dots I}$ is then split as the sum of tensor products of spinors of valence $\frac12$, and, then as explained in \cite{Penrose:1986fk} (vol. 1 p. 245, together with formula (4.9.4), (4.9.5) and (4.9.8)), using the fact:
\begin{equation*}
\nabla_{[B|A'}\nabla^{A'}_{|C]} u_{D}=-X_{BC}^{\phantom{BC}E}\phantom{}_{D}u_{E}
\end{equation*}
for any smooth section of $\ds_{D}$ (formula (4.9.8) in \cite{Penrose:1986fk}), we obtain:
\begin{equation*}
\nabla_{[B|A'}\nabla^{A'}_{|C]}\psi_{F\dots I}=-X_{BCF}^{\phantom{BCF}D}\psi_{DG\dots I}-X_{BCG}^{\phantom{BCG}D}\psi_{FD\dots I}-X_{BCI}^{\phantom{BCI}D}\psi_{FG\dots D}
\end{equation*} 
and finally:
\begin{eqnarray*}
\nabla_{BA'}\nabla^{FA'}\psi_{F\dots I}&=&-\frac{1}{2}\nabla_{HH'}\nabla^{HH'}\psi_{F\dots I}\\
&-&\eps^{FC}\big(X_{BCF}^{\phantom{BCF}D}\psi_{DG\dots I}-X_{BCG}^{\phantom{BCG}D}\psi_{FD\dots I}-X_{BCI}^{\phantom{BCI}D}\psi_{FG\dots D}\big)\\
&=&-\frac{1}{2}\nabla_{HH'}\nabla^{HH'}\psi_{F\dots I}\\ 
&-&X_{B}^{\phantom{B}F}\phantom{}_{F}^{\phantom{F}D}\psi_{DG\dots I}-X_{B}^{\phantom{B}F}\phantom{}_{G}^{\phantom{G}D}\psi_{FD\dots I}-\dots-X_{B}^{\phantom{B}F}\phantom{}_{I}^{\phantom{I}D}\psi_{FG\dots D}\text{\fin}
\end{eqnarray*}
As an obvious consequence of the definitions chosen for the Clifford multiplication and the Dirac operator on $\mathbb{E}$, the following proposition holds:
\begin{proposition}
The Dirac operator $\D$ on $\mathbb{E}$ and the Clifford multiplication by a vector field $v$ on $\Omega$ are respectively symmetric and skew symmetric with respect to the duality bracket $(,)_{\mathcal{D}'(\mathbb{E}),\mathcal{D}(\mathbb{E})}$ that is to say, for any $\phi$ and $\psi$ smooth sections of $\mathbb{E}$ with compact support :
$$
(\phi,\D \psi)_{\mathcal{D}'(\mathbb{E}),\mathcal{D}(\mathbb{E}))}=(\D\phi, \psi)_{\mathcal{D}'(\mathbb{E}),\mathcal{D}(\mathbb{E}))} \text{ and }(\phi,v\cdot\psi)_{\mathcal{D}'(\mathbb{E}),\mathcal{D}(\mathbb{E}))}=-(v\cdot\phi,\psi)_{\mathcal{D}'(\mathbb{E}),\mathcal{D}(\mathbb{E}))}.
$$
\end{proposition}

 All the methods that were developed for Dirac spinors can be used here, provided that we assume that we are working with $\mathbb{E}$-valued distributions. The structure of the fundamental solutions for the wave equation are the same:
\begin{equation*}
\tilde G^\pm_q(p)=\kappa^\pm_q(p)\tau_p(q)\delta(\Gamma_q(p))+V_q(p)H^\pm_q(p).
\end{equation*}
where $\tilde G^\pm$ is a bidistribution in $\mathcal{E}(\E)\boxtimes \mathcal{D}'(\E)$ which satisfies the wave equation:
\begin{equation*}
(\D^p)^2G^\pm_q(p)=\overline{\delta}_p(q),
\end{equation*}
$\overline{\delta}_p(q)$ being the Dirac mass in $p$. The application $\tau_p(q)$ satisfies the equation:
\begin{equation}\label{identite2}
(\tau_p(p),\phi)=\phi \text{ and } \nabla_i^q \Gamma_p(q)\nabla^i \tau_p(q)=0.
\end{equation}
\begin{remark} The functions $\tau$ and $V$ are more complex to write and we do not even try to do so, since the properties given by the equations (\ref{identite2}) are sufficient to conclude.
\end{remark}

A direct consequence of the previous remark is the following proposition:
\begin{proposition}\label{higherspinfundsol}
The Dirac operator $\D$ acting on sections of the fibre bundle $\E$ admits two fundamental solutions $G^\pm_p(q)=\D^q \tilde G^\pm_p(q)$, in $\mathcal{D}'(\E)\boxtimes \mathcal{D}'(\E)$, with, respectively, support in $\mathcal{C}^\pm(p)$, for any given $p$, which satisfy, in the sense of distributions:
\begin{equation*}
\D^q G^\pm_p(q)=\overline{\delta}_p(q).
\end{equation*}
\end{proposition}

Finally, we present the compacted spin coefficient formalism introduced by Penrose and Rindler in \cite{Penrose:1986fk}. Let $o^A, \iota^A$ be a given normalized spinor basis and consider the rescaling, for $\lambda$ in $\C$:
\begin{equation}\label{rescaling}
o^A\longmapsto\lambda o^A, \iota^A\longmapsto \frac{\iota^A}{\lambda} .
\end{equation}
\begin{definition}
A spinor $\phi$ is said to be of weight $(p,q)$ if and only if,  
under the transformation \eqref{rescaling}, $\phi$ is rescaled as:
$$
\phi\longmapsto \lambda^p\overline{\lambda}^q \phi
$$
The integer $\frac12(p-q)$ is the spin-weight of $\phi$ and $\frac12(p+q)$ is its boost-weight.
\end{definition}
We consider the Newman-Penrose tetrad $(l,n,m,\overline{m})$ associated with $o^A, \iota^A$. We define the differential operators with regard to these weights: let $\phi$ be a $(p,q)$ spinor. We define:
\begin{equation*}
\begin{array}{clc}
\mathfrak{p} \phi&=&\nabla_l\phi-p\epsilon\phi-q\overline\epsilon \phi\\
\eth' \phi&=&\nabla_{\overline{m}}\phi -p\alpha \phi+q\overline{\alpha}\phi
\end{array}
\end{equation*}
where $\epsilon=\iota^A\nabla_l o_A$ and $\alpha=\iota^A\nabla_{\overline{m}}o_A$.

Though the formalism of the Newman-Penrose tetrad will still be used, the usual notations $o^A, \iota^A$ for the basis spin basis are replaced by $\eps^A_0, \eps^A_1$. All the calculations will be performed using these notations. We must recall what is the link between these two notations: the spinor basis $(o^A, \iota ^A)$ is rewritten $(\eps^A_0, \eps^A_1)$, so that the dual basis is $(\eps_A^0, \eps_A^1)$ with $\eps_A^0=-\iota_A$ et $\eps_A ^1=o^A$. In this formalism, the spinors $\eps_A^\textbf{I}$ satisfy:
\begin{equation*}
\eps_A^\mathbf{J}\eps^A_\mathbf{I}=\delta^\mathbf{J}_\mathbf{I}.
\end{equation*}

Let now consider the field equation for spin $\frac{n}{2}$:
\begin{equation*}
\nabla^{AA'}\phi_{AB\dots F}=0
\end{equation*}
for a symmetric field $\phi_{AB\dots F}=\phi_{(AB\dots F)}$ with n indices; for $j$ in $\{0,1,\dots, n\}$, we define:
\begin{eqnarray*}
\phi_j&=&\underbrace{\eps^A_0\dots\eps^C_0}_{\text{n-j times}}\underbrace{\eps^D_1\dots\eps^F_1}_{\text{j times}}\phi_{AB \dots F}\\
&=&\underbrace{o^A\dots o^C}_{\text{n-j times}}\underbrace{\iota^D\dots\iota^F}_{\text{j times}}\phi_{AB \dots F}
\end{eqnarray*}
which are the only relevant components to calculate the field $\phi_{AB\dots F}$ wich can be written, because of its symmetry:
\begin{eqnarray*}
\phi_{A\dots F}&=&\sum_{j=0}^n \binom{n}{j}\phi_j\underbrace{\eps_{(A}^0\dots\eps_C^0}_{\text{n-j times}}\underbrace{\eps_D^1\dots\eps_{F)}^1}_{\text{j times}}\\
&=&\sum_{j=0}^n (-1)^{n-j}\binom{n}{j}\phi_j \underbrace{\iota_{(A}\dots \iota_C}_{\text{n-j times}}\underbrace{o_D\dots o_{F)}}_{\text{j times}}
\end{eqnarray*}
The quantity $\phi_j$ is a $(n-2r,0)$ scalar field. It is known to satisfy the following lemma (see \cite{Penrose:1986fk}, 4.12.42):
\begin{lemma}\label{scalarrelation}
Let $j$ be an integer in $\{2,\dots,n-1\}$.\\  Then $\phi_{j+1}$, $\phi_j$, $\phi_{j-1}$ and $\phi_{j-2}$ satisfy the following relation:
$$
\mathfrak{p}\phi_j-\eth'\phi_{j-1}=(j-1)\sigma'\phi_{j-2}-j  \tau '\phi_{j-1}+(n-j-1)\rho \phi_j-(n-j)\kappa\phi_{j+1}.
$$
\end{lemma}
\begin{remark} This is the more accurate way to write down the constraints equations on the cone, since the restriction to the tangential derivatives is obvious. 
\end{remark}

We conclude this section by giving the following relation between weighted scalars and differential forms (see \cite{Penrose:1986fk}, 4.14.70):
\begin{proposition}\label{formediff}
Let $\Sigma$ be a two dimensional spacelike closed surface with volume form $\mu_{\Sigma}$ and $\alpha$ a  $(1,-1)$ weighted spinor.\\
Then the integral of $\eth'\alpha$ over $\Sigma$ vanishes:
$$
\int_{\Sigma}\eth'\alpha\mu_{\Sigma}=0
$$
\end{proposition}
\subsection{Integral formula for spin $\frac{n}{2}$}

Let us consider the future characteristic Cauchy problem for the Dirac operator on $\E$:

\begin{equation}\label{ccphs2}
\left\{
\begin{array}{lcl}
\D u &=&0 \text{ on } \mathcal{J}^+(p_0)\\
u&=&\theta \text{ on } \mathcal{C}^+(p_0)
\end{array}
\right.,
\end{equation}
where $\theta$ is a smooth compactly supported function on the cone $ \mathcal{sC}^+(p_0)$. It must be noted that the problem (\ref{ccphs2}), contrary to the problem stated in (\ref{ccphs}), does not contain symmetry assumption. This assumption will be made afterwards to obtain the integral formula for (\ref{ccphs}).

By doing the same calculation as for proposition (\ref{representationformula}), a direct consequence of proposition (\ref{higherspinfundsol}) is the following integral formula:
\begin{proposition}\label{formuleintcomplete}
Let $u$ be a solution of (\ref{ccphs2}) in $\mathbb{E}$. 
Then $u$ can be written:
\begin{gather*}
u(q)=
\D^q\left(\int_{\sigma(q)}\left(\hat \nabla^p \Gamma_0\cdot u,\tilde U_p\right)_p\displaystyle{\frac{1}{4r_0r}\mu_{\sigma(q)}}\right)+\int_{\sigma(q)}\hat \nabla^qr_{0}\cdot  \big(\nabla^p \Gamma_{0}\cdot u,\tilde V_{p}(q)\big) 2r  \mu_{\sigma(q)}(p)\\+\int_{\mathcal{D}(q)}(\hat \nabla^p \Gamma_0 \cdot u,\D^q \tilde V_q)\mu_{\Gamma_0}(p).
\end{gather*}
\end{proposition}

The formula must now be simplified using the previous methods and a decomposition of the spinor $u$ on the same basis as in subsection \ref{geometricdataonthecone}: $u$ can be written:
\begin{equation*}
u=\phi_{\textbf{A}\dots \textbf{F}}\eps^{\textbf{A}}_A\dots \eps^{\textbf{F}}_F+\psi^{\textbf{A}'}_{\phantom{A'} \textbf{F}\dots\textbf{F}}\eps^{A'}_{\textbf{A}'}\eps^{\textbf{B}}_{B}\dots\eps^{\textbf{F}}_{F}.
\end{equation*}

The solution of 
\begin{equation}\label{ccphs6}
\left\{
\begin{array}{lcl}
\nabla^{AA'}u_{AB\dots F}&=&0 \text{ on } \mathcal{J}^+(p_0)\\
u_{AB\dots F}&=&\theta_{AB\dots F} \text{ on } \mathcal{C}^+(p_0)
\end{array}
\right.,
\end{equation}
obtained by projecting on $\ds_{A\dots F}$ the integral formula given in proposition \ref{formuleintcomplete}:
\begin{proposition}
Let $u_{A\dots F}$ be a solution of:
\begin{equation*}\label{ccphs4}
\left\{
\begin{array}{lcl}
\nabla^{AA'}u_{AB\dots F}&=&0 \text{ on } \mathcal{J}^+(p_0)\\
u_{AB\dots F}&=&\theta_{AB\dots F} \text{ on } \mathcal{C}^+(p_0)
\end{array}
\right.,
\end{equation*}
Then, $u_{A\dots F}$ can be written:
\begin{gather*}
u_{A\dots F}=\int_{\sigma(q)}\left(\frac{k_q(p)}{r}\right)\big(\nabla^p_l\phi_{0\textbf{B}\dots \textbf{F}}(p)-2\rho\phi_{0\textbf{B}\dots \textbf{F}}(p)\big)\big(\nabla^q_{AA'}r_0\big)\eps^{A'}_0(q)\eps^{\textbf{B}}_B(q)\dots\eps^{\textbf{F}}_F(q)\mu_{\sigma(q)}\\
+\int_{\sigma(q)}\phi_{0\textbf{B}\dots \textbf{F}}(p)\nabla_{AA'}\left(\left(\frac{k_{p}(q)}{r}\right)\eps^{A'}_0(q)\eps^{\textbf{B}}_B(q)\dots\eps^{\textbf{F}}_F(q)\right)\mu_{\sigma(q)}\\
+\int_{\sigma(q)}\hat \nabla^qr_{0}\cdot  \big(\nabla^p \Gamma_{0}\cdot u,\tilde V_{p}(q)\big)  \frac{\mu_{\sigma(q)}(p)}{2r }
 +\int_{\mathcal{V}_p}(\hat \nabla^p \Gamma_0 \cdot u,\D^p \tilde V_q)\mu_{\Gamma_0}(p)
\end{gather*}
\end{proposition}
\begin{remark}
Since our interest is in the singular part of the integral representation of the solution, we do not give a more explicit expression of the smooth part of the integral formula.
\end{remark}
\proof
The first step is to calculate the contraction $\hat \nabla^p \Gamma_0\cdot u$:
\begin{eqnarray*}
\hat \nabla^p \Gamma_0\cdot u&=&-2i\sqrt{2}r_0\eps^A_0\overline\eps^{A'}_0(\phi_{\textbf{A}\dots \textbf{F}}\eps^{\text{a}}_A\dots \eps^{\textbf{F}}_F)+2i\sqrt{2}r_0 \eps_A^0\overline\eps_{A'}^0 (\psi^{\textbf{A}'}_{\phantom{a'} \textbf{B}\dots\textbf{F}}\eps^{A'}_{\textbf{A}'}\eps^{\textbf{B}}_{B}\dots\eps^{\textbf{F}}_{F})\\
&=&2i\sqrt{2}r_0(-\eps^A_0\overline\eps^{A'}_0\phi_{\textbf{A}\dots \textbf{F}}\eps^{a}_A+\eps_A^1\overline\eps_{A'}^1 \psi^{\textbf{A}'}_{\phantom{a'} \textbf{B}\dots\textbf{F}}\eps^{A'}_{\textbf{A}'})\eps^{\textbf{B}}_B\dots \eps^{\textbf{F}}_F\\
&=&2i\sqrt{2}r_0(\phi_{0\textbf{B}\dots \textbf{F}}\eps^{A'}_0-\psi^{1'}_{\phantom{a'} \textbf{B}\dots\textbf{F}}\eps_A^1)\eps^{\textbf{B}}_B\dots \eps^{\textbf{F}}_F.
\end{eqnarray*}
Since $\tau_{p}(q)$ is obtained by doing a tensor product between an element of the spin basis at a point $p$ with the spinor obtained by parallely transporting this spinor along the geodesic from $p$ to $q$, which is an element of the spin basis at $q$,  the symplectic product $(\tau_p(q),\hat \nabla^p \Gamma_0\cdot u)$ realizes a switch between the variables $p$ and $q$:
\begin{equation*}
(\tau_p(q),\hat \nabla^p \Gamma_0\cdot u)=i\sqrt{2}r_0\big(\phi_{0\textbf{B}\dots \textbf{F}}(p)\eps^{A'}_0(q)-\psi^{1'}_{\phantom{a'} \textbf{B}\dots\textbf{F}}(p)\eps_A^1(q)\big)\eps^{\textbf{B}}_B(q)\dots \eps^{\textbf{F}}_F(q).
\end{equation*}
The interversion of the symbols $\int$ and $\D$ gives (we only make the calculation on $\ds_{A\dots F}$):
\begin{gather}
i\sqrt{2}\nabla_{AA'}^q\left(\int_{\sigma(q)}2i\sqrt{2}r_0\phi_{0\textbf{B}\dots \textbf{F}}(p)\eps^{A'}_0(q)\eps^{\textbf{B}}_B(q)\dots \eps^{\textbf{F}}_F(q)\mu_{\Gamma_0,\Gamma_q}\right)\\=-\int_{\sigma(q)}\phi_{0\textbf{B}\dots \textbf{F}}(p)\nabla_{AA'}^q\left(\frac{k_q(p)}{r}\eps^{A'}_0(q)\eps^{\textbf{B}}_B(q)\dots \eps^{\textbf{F}}_F(q)\right)\mu_{\sigma(q)} \\
-\int_{\sigma(q)}\frac{k_q(p)}{r}(\nabla^p_l\phi_{0\textbf{B}\dots \textbf{F}}(p))\big(\nabla^q_{AA'}r_0\big)\eps^{A'}_0(q)\eps^{\textbf{B}}_B(q)\dots\eps^{\textbf{F}}_F(q)\mu_{\sigma(q)}\\
-\int_{\sigma(q)}\nabla^p_l\left(\frac{k_q(p)}{r}\right)(\phi_{0\textbf{B}\dots \textbf{F}}(p))\big(\nabla^q_{AA'}r_0\big)\eps^{A'}_0(q)\eps^{\textbf{B}}_B(q)\dots\eps^{\textbf{F}}_F(q)\mu_{\sigma(q)}\\
+\int_{\sigma(q)}2\rho\frac{k_q(p)}{r}(\phi_{0\textbf{B}\dots \textbf{F}}(p))\big(\nabla^q_{AA'}r_0\big)\eps^{A'}_0(q)\eps^{\textbf{B}}_B(q)\dots\eps^{\textbf{F}}_F(q)\mu_{\sigma(q)}
\end{gather}
which can be simplified in:
\begin{gather}
i\sqrt{2}\nabla_{AA'}^q\left(\int_{\sigma(q)}2i\sqrt{2}r_0\phi_{0\textbf{B}\dots \textbf{F}}(p)\eps^{A'}_0(q)\eps^{\textbf{B}}_B(q)\dots \eps^{\textbf{F}}_F(q)\mu_{\Gamma_0,\Gamma_q}\right)\label{integral225}\\
=-\int_{\sigma(q)}\left(\frac{k_q(p)}{r}\right)\big(\nabla^p_l\phi_{0\textbf{B}\dots \textbf{F}}(p)-2\rho\phi_{0\textbf{B}\dots \textbf{F}}(p)\big)\big(\nabla^q_{AA'}r_0\big)\eps^{A'}_0(q)\eps^{\textbf{B}}_B(q)\dots\eps^{\textbf{F}}_F(q)\mu_{\sigma(q)}\label{rest}\\
-\int_{\sigma(q)}\phi_{0\textbf{B}\dots \textbf{F}}(p)\nabla_{AA'}\left(\left(\frac{k_{q}(p)}{r}\right)\eps^{A'}_0(q)\eps^{\textbf{B}}_B(q)\dots\eps^{\textbf{F}}_F(q)\right)\mu_{\sigma(q)}\label{sumeps}
\end{gather}
The next part of the integral formula is exactly the same as in the case of the Weyl-Dirac spinors and is obtained in a similar way.\fin

Finally, to obtain a solution of the full problem with symmetry, it is sufficient to symmetrize the unprimed indices in the formula; we then give a representation, when the problem \eqref{ccphs} makes sense (i.e with adequat restrictions on the curvature for spin greater than $\frac32$):
\begin{theorem}\label{partiesingccphscourbe}
Let $u_{A\dots F}$ be a solution of the symmetrized characterictic Cauchy problem
\begin{equation}\label{ccphs5}
\left\{
\begin{array}{lcl}
\nabla^{AA'}u_{AB\dots F}&=&0 \text{ on } \mathcal{J}^+(p_0)\\
u_{0}&=&\theta_{AB\dots F} \text{ on } \mathcal{C}^+(p_0)
\end{array}
\right.,
\end{equation}
where $u_{AB\dots F}$ satisfies the symmetry conditions:
$
u_{AB\dots F}=u_{(AB\dots F)}.\\
$
Then the singular part of the integral representation of $u_{A\dots F}$, that is to say the part supported on the intersection of the cone is given by the formula:
\begin{gather*}
\int_{\sigma(q)}\left(\frac{k_q(p)}{r}\right)\big(\nabla^p_l\phi_{0\textbf{B}\dots \textbf{F}}(p)-2\rho\phi_{0\textbf{B}\dots \textbf{F}}(p)\big)\big(\nabla^q_{AA'}r_0\big)\eps^{A'}_0(q)\eps^{\textbf{B}}_B(q)\dots\eps^{\textbf{F}}_{F}(q)\mu_{\sigma(q)}\\
\int_{\sigma(q)}\phi_{0\textbf{B}\dots \textbf{F}}(p)\nabla_{AA'}\left(\left(\frac{k_{p}(q)}{r}\right)\eps^{A'}_0(q)\eps^{\textbf{B}}_B(q)\dots\eps^{\textbf{F}}_{F}(q)\right)\mu_{\sigma(q)}
\end{gather*}
\end{theorem}

\subsection{Integral formula for spin $\frac{n}{2}$ in the flat case}

This subsection is devoted to the recovery of the Penrose formula; with the same notations as before, the following proposition holds:
\begin{proposition}[Integral formula for the flat case for spin $\frac{n}{2}$]
Let $\phi_{A\dots F}$ be a solution of \eqref{ccphs5} on the Minkowski space time.\\
Then $\phi$ can be written:
\begin{equation*}
\phi_{A\dots F}=(-1)^{n}\int_{\sigma(q)}(\nabla_l \phi_0-(n+1)\rho \phi_0)\iota_{A}\dots\iota_F\frac{\mu_{\sigma(q)}}{2\pi r}
\end{equation*}
\end{proposition}
\begin{remark} The formula which is given here agrees with the one obtained by Penrose in \cite{p80} (formula 4.9). The $(-1)^n$ comes from the fact that Penrose chooses the convention: 
$$
\iota^{A}\longmapsto-\iota^{A}
$$
because of the different choice of normalization (formula (4.7), \emph{op. cit.}):
$$
\iota_{A}o^{A}=1
$$
whereas our convention is:
$$
o_{A}\iota^{A}=1.
$$
\end{remark}

\proof: We summarize the geometric elements required to perform the calculation: 
\begin{remark}\label{constmink}
We recall the main properties of the spinor basis which was constructed in section \ref{geometricdataonthecone}:
\begin{enumerate}
\item the spinors $o^A$ and $\iota^A$ are constant along a generator of the cone $\mathcal{J}^+(p_0)$, so that
the spin coefficients corresponding to the derivatives  of $o^A,\iota^A$ along the vector $l^a=o^A\overline{o}^A$, $\kappa, \eps, \tau'$ are zero;
\item furthermore, for $q$ in $\mathcal{J}^+(p_0)$, the basis $(o^A, \iota^A)$ is parallely transported along the integral curves of $\iota^A\overline{\iota}^{A'}$ and so, in the flat case, is constant along the null generators of the cone $\mathcal{J}^-(q)$; 
\item the derivatives along $m$  of $o^A$ and $\iota^A$ are calculated explicitly (see \cite{Penrose:1986fk}, 4.12.28):
$$
\eth' o^A=-\rho \iota^A \text{ and } \eth' \iota^A=-\sigma' o^A ;
$$
\item the derivatives of $\iota^A$ and $r$ can be explicitly calculated by differentiating the relation:
$$
\vec{p_0p}^a=r_0l^a+r\iota^A\overline{\iota}^{A'}
$$
for any $p$ in $\mathcal{J}^+(p_0)$. Their derivatives are: 
\begin{equation}\label{rderivative}
\nabla_{BB'}\iota^A=-\frac{1}{r}\iota_B\overline{o}_{B'}o^A \text{ and } \nabla_{AA'}r=o_A\overline{o}_{A'}
\end{equation}
and, consequently, the only non-vanishing derivative of $\iota^A$ is 
$$
\nabla_m \iota^A=\frac{1}{r}o^A
$$
and the spin coefficients 
$$
\tau'=-\iota^A\nabla_l\iota_A, \sigma'=-\iota^A\nabla_{\overline{m}}\iota_A, \beta'=-\alpha=-\iota^A\nabla_{\overline{m}}\iota_A \text{ and } \beta=-\alpha'=-\iota^A\nabla_{m}\iota_A
$$
vanish.
\item Using equations \eqref{rderivative} and since $\iota^A$ is a $(-1,0)$-spinor and $r$ is a $(1,1)$ scalar, the following derivatives vanish:
$$
\eth'\iota^A=0 \text{ and } \eth'r=0.
$$
\end{enumerate} 
\end{remark}

For the sake of clarity, the calculation is first performed for the Maxwell equations and then for the arbitrary spin. The first step is to write the Maxwell equations
\begin{equation*}
\nabla^{AA'}\phi_{AB}=0
\end{equation*}
as
\begin{equation}\label{npmaxwell}
\begin{array}{rcl}
\nabla_l\phi_1-\nabla_{\overline{m}}\phi_0&=&(\pi-2\alpha)\phi_0+2\rho \phi_1-\kappa \phi_2\\
\nabla_l\phi_2-\nabla_{\overline{m}}\phi_1&=&-\lambda \phi_0+2\pi \phi_1+(\rho-2\eps)\phi_2\\
\nabla_m\phi_1-\nabla_n \phi_0&=&(\mu-2\gamma)\phi_0+2\tau \phi_1-\sigma \phi_2\\
\nabla_m\phi_2-\nabla_n \phi_1&=&-\nu\phi_0+2\mu\phi_1+(\tau-2\beta)\phi_2
\end{array}
\end{equation}
with the convention $\phi_{00}=\phi_0$, $\phi_{10}=\phi_1$ and $\phi_{11}=\phi_2$.  We then consider the singular part:

\begin{gather*}
\int_{\sigma(q)}\left(\frac{k_q(p)}{r}\right)\big(\nabla^p_l\phi_{0\textbf{b}}(p)-2\rho\phi_{0\textbf{b}}(p)\big)\big(\nabla^q_{AA'}r_0\big)\eps^{A'}_0\eps^{\textbf{b}}_B\mu_{\sigma(q)}\\
=\int_{\sigma(q)}\left(\frac{k_q(p)}{r}\right)\big(\nabla^p_l\phi_{0\textbf{b}}(p)-2\rho\phi_{0\textbf{b}}(p)\big)\big(\iota_A\iota_{A'}\big)\eps^{A'}_0\eps^{\textbf{b}}_B\mu_{\sigma(q)}\\
=-\int_{\sigma(q)}\left(\frac{k_q(p)}{r}\right)\big(\nabla^p_l\phi_{0\textbf{b}}(p)-2\rho\phi_{0\textbf{b}}(p)\big)\iota_A\eps^{\textbf{b}}_B\mu_{\sigma(q)}\\
=-\int_{\sigma(q)}\left(\frac{k_q(p)}{r}\right)\big(\nabla^p_l\phi_{00}(p)-2\rho\phi_{00}(p)\big)\iota_A\iota_B\mu_{\sigma(q)}\\-\underbrace{\int_{\sigma(q)}\left(\frac{k_q(p)}{r}\right)\big(\nabla^p_l\phi_{01}(p)-2\rho\phi_{01}(p)\big)\iota_Ao_B\mu_{\sigma(q)}}_{=B}
\end{gather*}
with $\overline{\kappa}=\frac{1}{2\pi r}$. Using the first Maxwell equation \eqref{npmaxwell}, and since, for the choice of basis which was previously done, the spin coefficients $\kappa=o^A\nabla_l o_A$ and $\pi=-\iota^A\nabla_{l}\iota_A$ vanish, we obtain:
\begin{gather*}
B=\int_{\sigma(q)}\left(\frac{k_q(p)}{r}\right)\big(\nabla^p_l\phi_{01}(p)-2\rho\phi_{01}(p)\big)\iota_Ao_B\mu_{\sigma(q)}=
\int_{\sigma(q)}\frac{1}{2 \pi r}\big( \nabla^p_{\overline{m}}\phi_{00}-2\alpha\phi_{00}\big)\iota_Ao_B\mu_{\sigma(q)}\\=
\int_{\sigma(q)} \nabla_{\overline{m}}\left(\frac{\phi_{00}\iota_Ao_B}{2\pi r}\right)-2\alpha\frac{\phi_{00}\iota_Ao_B}{2\pi r}+\nabla_{\overline{m}}r\frac{\phi_{00}\iota_Ao_B}{2\pi r^2}-\frac{\phi_{00}\iota_A\nabla_{\overline{m}}(o_B)}{2\pi r}-\frac{\phi_{00}\nabla_{\overline{m}}(\iota_A)o_B}{2\pi r}\mu_{\sigma(q)}
\end{gather*}
Since 
\begin{equation*}
\nabla^p_{\overline{m}}r=\iota^B\overline{o}^{B'},\ \nabla^p_{BB'}r=\iota^B\overline{o}^{B'}o_{B}\overline{o}_{B'} \text{ and } \nabla^{p}_{\overline{m}}o_B=-\rho\iota_B,
\end{equation*}
and 
\begin{equation*}
\int_{\sigma(q)}\nabla^p_{\overline{m}}\left(\frac{\phi_{00}\iota_Ao_B}{2\pi r}\right)-2\alpha\frac{\phi_{00}\iota_Ao_B}{2\pi r}\mu_{\sigma(q)}=0,
\end{equation*}
by Stoke's theorem (cf. (4.14.70) in \cite{Penrose:1986fk}; it is possible to reinterpret this expression using the compacted spin coefficient formalism), we obtain:
\begin{equation*}
B=\int_{\sigma(q)}\rho\frac{\phi_{00}}{2 \pi r}\iota_A\iota_B\mu_{\sigma(q)}.
\end{equation*}

We finally have the expected integral formula for the Maxwell equation:
\begin{gather*}
\int_{\sigma(q)}\left(\frac{k_q(p)}{r}\right)\big(\nabla^p_l\phi_{0\textbf{b}}(p)-2\rho\phi_{0\textbf{b}}(p)\big)\big(\nabla^q_{AA'}r_0\big)\eps^{A'}_0(q)\eps^{\textbf{b}}_B(q)\mu_{\sigma(q)}\\=\int_{\sigma(q)}(\nabla_l\phi_{00}-2\rho)\iota_A\iota_B\frac{\mu_{\sigma(q)}}{2\pi r}-B\\
=\int_{\sigma(q)}(\nabla^p_l\phi_{00}-3\rho)\iota_A\iota_B\frac{\mu_{\sigma(q)}}{2\pi r}.
\end{gather*}

The first step of the general proof is to notice that, as in the Maxwell case, the only remaining term in the flat case is the equation \eqref{rest} since the term \eqref{sumeps} vanishes. So the simplification of the equation \eqref{rest} can be done using the same methods.

A direct consequence of remark \ref{constmink} is that the relation given in lemma \ref{scalarrelation} is considerably simplified:
\begin{equation}\label{simprelation}
\forall j\in \{1, n-1\}, \mathfrak{p}\phi_j-2\rho \phi_j=\eth'\phi_{j-1}+(n-j-1)\rho \phi_j.
\end{equation}

In the Minkowski case, the only non-vanishing term in the integral formula is the following:
\begin{equation*}
\int_{\sigma(q)}\left(\frac{k_q(p)}{r}\right)\big(\nabla^p_l\phi_{0\textbf{b}\dots \textbf{f}}(p)-2\rho\phi_{0\textbf{b}\dots \textbf{f}}(p)\big)\big(\nabla^q_{(AA'}r_0\big)\eps^{A'}_0(q)\eps^{\textbf{b}}_B(q)\dots\eps^{\textbf{f}}_{F)}(q)\mu_{\sigma(q)}
\end{equation*}
which can be simplified in a flat space as: 
\begin{equation*}
-\int_{\sigma(q)}\frac{1}{2\pi r}\big(\nabla^p_l\phi_{0\textbf{B}\dots \textbf{F}}(p)-2\rho\phi_{0\textbf{F}\dots \textbf{F}}(p)\big)\iota_{(A}(q)\eps^{\textbf{B}}_B(q)\dots\eps^{\textbf{F}}_{F)}(q)\mu_{\sigma(q)}\\
\end{equation*}

Consider the generic term in this sum: let $j$ be an integer in $\{1,\dots,n-1\}$:
$$
\int_{\sigma(q)}\frac{1}{2\pi r}(\nabla_l \phi_j-2\rho)\iota_{(A}(q)\eps^{\textbf{B}}_{B}(q)\dots\eps^{\textbf{F}}_{F)}(q)\mu_{\sigma(q)}
$$
Since $\phi_j$ is obtained by contracting j times $\phi_{A\dots F}$ with $\iota^A$ and $n-j$ times with $o^A$, it means that there is exactly j times $o_A$ and $n-j-1$ times $-\iota_A$ in the list $\eps^{\textbf{B}}_B(q)\dots\eps^{\textbf{F}}_{F}(q)$; since the sum is symmetric, it can be written:
\begin{gather*}
\int_{\sigma(q)}\frac{1}{2\pi r}(\nabla_l \phi_j-2\rho \phi_j)\underbrace{\iota_{(A}(-\iota_{B})\dots(-\iota_C)}_{n-j \text{ terms}} \underbrace{o_D\dots o_{F)}}_{j \text{ terms}} \mu_{\sigma(q)}.
\end{gather*}
Using equation \eqref{simprelation}, it becomes:
\begin{gather*}
\int_{\sigma(q)}\frac{1}{2\pi r}(\nabla_l \phi_j-2\rho\phi_j)\iota_{(A}\dots\iota_Co_D\dots o_{F)}\mu_{\sigma(q)}\\=\int_{\sigma(q)}\frac{1}{2\pi r}(n-j-1)\rho \phi_j\iota_{(A}\dots\iota_Co_D\dots o_{F)} \mu_{\sigma(q)}+\int_{\sigma(q)}\frac{1}{2\pi r}\eth'(\phi_{r-1})\iota_A\dots\iota_Co_D\dots o_{F)} \mu_{\sigma(q)}.
\end{gather*}
Using remarks \ref{constmink}, the last integral is written as a difference: 
\begin{gather}
\int_{\sigma(q)}\frac{1}{2\pi r}\eth'(\phi_{j-1})\iota_{(A}\dots\iota_Co_D\dots o_{F)} \mu_{\sigma(q)}=\int_{\sigma(q)}\eth'\left(\frac{\phi_{j-1}\iota_{(A}\dots\iota_Co_D\dots o_{F)}}{2\pi r}\right)\mu_{\sigma(q)}\label{3342}\\
+j\int_{\sigma(q)}\rho\phi_{j-1}\underbrace{\iota_{(A}\dots\iota_D}_{n-j-1}\underbrace{o_E\dots\iota_{F)}}_{j-1}\frac{\mu_{\sigma(q)}}{2\pi r}.
\end{gather}
It has already been noted that:
\begin{enumerate}
\item $r$ is $(1,1)$ scalar; 
\item $\phi_{j-1}$ is a $(n-2j+2,0)$ scalar;
\item $\iota_{(A}\dots\iota_Co_D\dots o_{F)}$ is a $(2j-n,0)$ spinor.
\end{enumerate}
As a consequence, the term integrated in the left-hand side of equation \eqref{3342} and under the derivation $\eth'$ is $(1,-1)$ spinor. In order to apply lemma \ref{formediff}, this spinor is contracted with $n$ constant arbitrary spinors; this gives:
$$
\int_{\sigma(q)}\eth'\left(\frac{\phi_{j-1}}{2\pi r}\underbrace{\iota_{(A}\dots\iota_C}_{n-j\text{ terms}} \underbrace{o_D\dots o_{F)}}_{j\text{ terms}}\right)\mu_{\sigma(q)}=0.
$$
Finally, we obtain:
\begin{gather}
\int_{\sigma(q)}(\nabla_l \phi_j-2\rho \phi_r)\underbrace{\iota_{(A}\dots\iota_C}_{n-j\text{ terms}} \underbrace{o_D\dots o_{F)}}_{j \text{ terms}}\frac{\mu_{\sigma(q)}}{2\pi r}=\\(n-j-1)\int_{\sigma(q)}\rho \phi_j\underbrace{\iota_{(A}\dots\iota_C}_{n-j}\underbrace{o_D\dots o_{F)} }_j\frac{\mu_{\sigma(q)}}{2\pi r}+j\int_{\sigma(q)}\rho\phi_{j-1}\underbrace{\iota_{(A}\dots\iota_D}_{n-j+1}\underbrace{o_E\dots\iota_{F)}}_{j-1}\frac{\mu_{\sigma(q)}}{2\pi r}
\end{gather}

Theses terms are added to obtain the complete expression of the integral formula:
\begin{gather*}
\int_{\sigma(q)}(\nabla_l \phi_{0\textbf{b\dots f}}-2\rho)\iota_{(A}(q)\eps^{\textbf{b}}_{B}(q)\dots\eps^{\textbf{f}}_{F)}(q)\frac{\mu_{\sigma(q)}}{2\pi r}\\
=\sum_{j=0}^{n-1}(-1)^{n-j-1}\binom{n-1}{j}\int_{\sigma(q)}(\nabla_l \phi_j-2\rho \phi_j)\underbrace{\iota_{(A}\dots\iota_C}_{n-j \text{ terms}} \underbrace{o_D\dots o_{F)}}_{j \text{ terms}}\frac{\mu_{\sigma(q)}}{2\pi r} \\
=\int_{\sigma(q)}(\nabla_l \phi_0-2\rho \phi_0)\iota_{A}\dots\iota_F\frac{\mu_{\sigma(q)}}{2\pi r}+\\
\sum_{j=1}^{n-1}(-1)^{n-j-1}(n-j-1)\binom{n-1}{j}\int_{\sigma(q)}\rho \phi_j\underbrace{\iota_{(A}\dots\iota_C}_{n-j}\underbrace{o_D\dots o_{F)} }_j\frac{\mu_{\sigma(q)}}{2\pi r}\\
+\sum_{j=1}^{n-1}(-1)^{n-j-1}j\binom{n-1}{j}\int_{\sigma(q)}\phi_{j-1}\underbrace{\iota_{(A}\dots\iota_D}_{n-j-1}\underbrace{o_E\dots\iota_{F)}}_{j-1}\frac{\mu_{\sigma(q)}}{2\pi r}.
\end{gather*} 
The sum is split in two and reindexed:
\begin{gather*}
\sum_{j=1}^{n-1}(-1)^{n-j-1}(n-j-1)\binom{n-1}{j}\int_{\sigma(q)}\rho \phi_j\underbrace{\iota_{(A}\dots\iota_C}_{n-j}\underbrace{o_D\dots o_{F)} }_j\frac{\mu_{\sigma(q)}}{2\pi r}\\
+\sum_{j=1}^{n-1}(-1)^{n-j-1}j\binom{n-1}{j}\int_{\sigma(q)}\phi_{j-1}\underbrace{\iota_{(A}\dots\iota_D}_{n-j-1}\underbrace{o_E\dots\iota_{F)}}_{j-1}\frac{\mu_{\sigma(q)}}{2\pi r}=\\
\sum_{j=1}^{n-2}(-1)^{n-j-1}\underbrace{\left((n-j-1)\binom{n-1}{j}-(j+1)\binom{n-1}{j+1}\right)}_{=0}\int_{\sigma(q)}\rho \phi_j\underbrace{\iota_{(A}\dots\iota_C}_{n-j}\underbrace{o_D\dots o_{F)} }_j\frac{\mu_{\sigma(q)}}{2\pi r}\
\end{gather*}
The remaining terms are then:
\begin{gather*}
(-1)^{n-1}\int_{\sigma(q)}(\nabla_l \phi_0-2\rho \phi_0)\iota_{A}\dots\iota_F\frac{\mu_{\sigma(q)}}{2\pi r}-(-1)^{n-1}\binom{n-1}{1}\int_{\sigma(q)}\rho \phi_0\iota_{A}\dots\iota_F\frac{\mu_{\sigma(q)}}{2\pi r}
\end{gather*}
and the integral formula is, because of the antisymmetry of the symplectic product:
\begin{equation}
\phi_{A\dots F}=(-1)^{n}\int_{\sigma(q)}(\nabla_l \phi_0-(n+1)\rho \phi_0)\iota_{A}\dots\iota_F\frac{\mu_{\sigma(q)}}{2\pi r}
\end{equation}
is proved.\fin

\newpage
\textbf{Concluding remarks}
\begin{enumerate}
\item Klainerman-Rodnianski state in \cite{MR2339803} that a $C^2$ metric (or a square-integrable Riemann curvature) is sufficient to write the singular part of the Kirchoff-Sobolev parametrix for the Einstein equations. It is expected that such a regularity will not prevent the use of this method for the arbitrary spin Dirac equation.
\item The construction of the representation formula is flexible enough to be used with other fiber bundles. Provided that the correct geometric hypotheses are stated for the manifold, such a representation can thus be obtained for the Rarita-Schwinger (or gravitino) equations.
\item Chrusciel-Shatah obtained in \cite{AJM-1-3-530-548} $L^2$-estimates for the Yang-Mills equations. We hope that such estimates can be obtained using the Friedlander construction of an integral formula. Nonetheless, it must be noted that they intensively used the gauge freedom which exists for the Yang-Mills equation: they used both the Cronström gauge (to obtain pointwise estimates) and the temporal gauge (to obtain estimates on spacelike slices). Similar estimates for the Dirac equations could help to explain, for instance, the loss of regularity observed in the characteristic Cauchy problem in \cite{ha06} (section 6).
\end{enumerate}

This work was partially supported by the ANR project JC0546063 ``Equations hyperboliques
dans les espaces-temps de la relativit\'e g\'en\'erale : Diffusion et r\'esonances.''
\bibliographystyle{siam}
\bibliography{scatt}

\end{document}